\documentclass[11pt]{article}
\usepackage{amsthm,amsmath,amssymb,amsfonts}
\usepackage{url}

\newcommand{\dx}{{\, \rm d}x}

\newcommand{\Div}{{\rm div}\,}

\newtheorem{thm}{Theorem}
\newtheorem{cor}{Corollary}

\newtheorem{lem}{Lemma}

\newtheorem{df}{Definition}
\newtheorem{rmk}{Remark}

\newcommand{\T}{\mathcal{T}}

\newcommand{\vr}{\varrho}

\newcommand{\vt}{\vartheta}

\newcommand{\vu}{\vc{u}}

\newcommand{\vc}[1]{{\bf #1}}

\newcommand{\F}[1]{$\mathbb{#1}$}
\newcommand{\Grad}{\nabla}

\newcommand{\tn}[1]{\mbox {\F #1}}

\newcommand{\dt}{\, {\rm d} t }

\newcommand{\dxdt}{\dx \dt}

\newcommand{\intO}[1]{\int_{\Omega} #1\dx}

\newcommand{\ep}{\varepsilon}

\begin{document}

\title{Continuity equation and vacuum regions { in} compressible flows}
\author{Anton\'\i n  Novotn\'y$^1$ and  Milan Pokorn\' y$^2$}
\maketitle

\bigskip

\centerline{$^1$ Institut de Math\'ematiques de Toulon, EA 2134}
\centerline{BP20132, 83957 La Garde, France}
\centerline{e-mail: {\tt novotny@univ-tln.fr}}

\centerline{$^{2}$ Charles University, Faculty of Mathematics and Physics}
\centerline{Mathematical Inst. of Charles University} 
\centerline{Sokolovsk\' a 83, 186 75 Prague 8, Czech Republic}
\centerline{e-mail: {\tt pokorny@karlin.mff.cuni.cz}}
\vskip0.25cm
\begin{abstract}
We investigate the creation and properties of eventual vacuum regions in the weak solutions of the continuity equation, in general, and
in the weak solutions of compressible Navier--Stokes equations, in particular. The main results are based on the analysis of renormalized solutions
to the continuity and pure transport equations and their inter-relations which are of independent interest.
\end{abstract}


\smallskip


\section{Introduction}

{ In this paper we consider evolution of the couple $(\vr,\vu)= (\vr(t,x),\vu(t,x))$---(density, velocity) of the compressible fluid---over the time interval $I$,} 
$I=(0,T)$, $T>0$, $t\in \overline I$
in a bounded domain $\Omega\in R^d$, $d\ge 2$, $x\in \Omega$. We concentrate on the question of the creation of vacuum regions $\{x\in\Omega|\vr(t,x)=0\}$ in this flow.
This is one of important open questions in the mathematical fluid mechanics of compressible fluids. It is closely connected to the question of regularity of solutions to the compressible Navier--Stokes equations. If the density is initially bounded away from zero, for weak solutions it is not excluded that the vacuum may appear in finite time. 

{ We  show that if this happens it must happen in a sense smoothly. More precisely, the measure of the set, where the density may be equal to zero, is continuous in time, or, in the other words, the vacuum (if any) creates and evolutes continuously in time and  the vacuum of positive measure cannot
appear instantaneously.}  The exact formulation of this result is { presented} in Theorem \ref{t2}. 

{ More interesting and intriguing is the second result.  It translates as follows: Assume that $(\vr,\vu)$ is a (standard) weak solution to the compressible Navier--Stokes equations.  Then whatever distributional solution $R$ with a small additional regularity (specified in (\ref{clR-})) of the continuity equation  with the same velocity $\vu$ we take (whatever arbitrary its initial data are!),  $R$ must develop at any time $t$ a vacuum region $\{x|R(t,x)=0\}$ that includes the vacuum of 
$\vr(t)$, i.e. $\{x|\vr(t,x)=0\}$ is contained in the vacuum set of the function $R$. This result definitely pleads for a non-existence of vacuum 
in compressible flows at least in many physically reasonable situations.} The exact formulation of this result is given in Theorem \ref{t3} and its Corollaries 
\ref{cor1}, \ref{cor2}.

{ On the other hand, it is important to recall that if the velocity field $\vu \in L^2(0,T; W^{1,2}(\Omega;R^d))$  (this is the generic situation for flows
of  Newtonian fluids with constant viscosities), there is  no direct way of constructing solutions $R$ to the continuity equation with the given velocity 
unless  {${\rm div}\, \vu\in L^1(I;L^\infty(\Omega))$---cf. DiPerna--Lions} \cite[Proposition II.1]{DL}. Indeed, existence of  solutions to the continuity equation 
with the transporting velocity fields  in spaces $L^1(I;W^{1,p}(\Omega;R^d))$, $p\in [1,\infty)$ only, is, in general, 
an open problem.} 


The  conclusions of our paper described above are based on nowadays classical results and techniques for the continuity and transport equation 
that have been forged within the process of the development of the existence theory for weak solutions to the compressible { Navier--Stokes} equations
and recently also for the mixtures of compressible fluids. They are all inspired by the classical regularization technique  
implemented to the investigation of transport equations with transport coefficients in Sobolev spaces in the seminal work of { DiPerna--Lions} \cite{DL}.
(The spaces needed for the results in \cite{DL} are those needed in the Friedrichs lemma about commutators with $\alpha=\infty$, $p=1$, 
cf. Lemma \ref{l2}.) Some of them are valid only within the functional setting of the transport theory \cite{DL} (namely those dealing
with extension of distributional solutions to weak solutions (up to the boundary), time integration of weak or distributional solutions and
passage from distributional or weak solutions to renormalized distributional or weak solutions)\footnote{The various notions of solutions used in 
the above text are rigorously defined in Section { \ref{S2}.}}. They are formulated in Subsection \ref{M2.2}
in Theorems \ref{th3} and \ref{th4}.  
Some of them, namely those valid for the renormalized solutions, must go beyond the transport theory \cite{DL}  
(in the sense that the transporting velocity belongs still
to Sobolev spaces but one requires less summability of the solution then the summability required in \cite{DL}){---in order to get stronger results with respect to the constitutive laws of pressure in the applications to compressible Navier--Stokes equations.} 
(This is notably the case of Theorems \ref{th1}
and \ref{th2} in Subsection \ref{M2.1}). Indeed, all available constructions of weak solutions to the compressible { Navier--Stokes} equations
provide a couple $(\vr,\vu)$ which satisfies the continuity equation in the renormalized sense.
The latter results are often formulated in the mathematical literature in a particular 
functional setting applicable to the concrete situation without ambition to full generality, see  Lions \cite{L4}, Feireisl \cite{FeBook} and
\cite{FeNo_book}, \cite{NoSt} if we limit ourselves to the monographs only. Our aim is to provide generalization and synthesis
of the results we need and  prove them in their full generality, either for the sake of completeness or if we could not find a reliable exhausting 
reference. 

A new approach to the compactness in the compressible Navier-Stokes equations allowing to treat some other physically different situations
then \cite{FeNoPe}, \cite{FeBook}, \cite{FeNo_book} 
has been introduced by Bresch, Jabin \cite{BrJa}, deriving, in particular a "$\log\log$ estimate" for the Friedrichs type commutator,
\cite[Theorem 2.3.6]{BrJa1} in the { DiPerna--Lions} functional framework. This theory does not allow to go beyond the DiPerna--Lions
functional setting and seems at the time being so far in-exploitable for our purpose.

Among the main auxiliary questions which has to be answered in order to apply the theory of transport equations
 to the compressible fluid dynamics  in general, and
to the investigation of the vacuum states, in particular, are the following: 
\begin{enumerate}
\item What are the least conditions imposed on the transporting velocity $\vu$ (in terms of Sobolev spaces) and solution $\vr$ of the
continuity equation (in terms of Lebesgue spaces)  allowing to pass from renormalized distributional or weak solutions to time integrated weak solutions?
The answers to these questions are subject of Theorems \ref{th1} and \ref{th2}. 
\item  What are the least conditions on the couple $(\vr,\vu)$ (in the same functional setting) to pass from distributional solutions of the continuity or pure transport equations  to the weak (up to the boundary)
solutions (eventually to the renormalized weak solutions), and from distributional or weak solutions  to their time integrated
counterparts (eventually to the renormalized time integrated counterparts)?  The answer to these questions are given in Theorems \ref{th3} and \ref{th4}.
\item How are interconnected solutions of pure transport equation and continuity equations? and what does this interconnection 
imply for the formation of vacuum in the compressible flows? The first question is treated in Theorem \ref{th5}. The last question is
object of Theorems \ref{t2}, \ref{t3} and their Corollaries \ref{cor1}, \ref{cor2}. 
\end{enumerate}

It is to be { noticed that the  conditions mentioned in Items  1.-3.} determine in large extend the admissible constitutive laws 
in the theory of weak solutions to compressible { Navier--Stokes} equations, \cite{FeNoPe}, \cite{FeNo_book},  \cite{Fe2002},
\cite{BrJa}. The usefulness of the subject of Item 4. was firstly discovered 
in connection with the investigation of weak solutions of systems describing compressible mixtures, see \cite{3MNPZ}, \cite{NoPo},
\cite{AN}, \cite{VWY}.

Our approach is exclusively Eulerian. The Lagrangian approach (dealing with characteristics of the vector field $\vu$
rather than with the transport equation, and translating them afterwards to the Eulerian vocabulary)  introduced in seminal paper of Ambrosio \cite{Ambr} 
allows to extend some results of \cite{DL}
(namely those related to existence, uniqueness and passage from distributional or weak to renormalized distributional or weak solutions) 
to $L^1(I;BV(\Omega;R^d))$ vector fields\footnote{The space $BV(\Omega)$ is the space of functions with bounded variations.} with divergence always in $L^1(I;L^\infty(\Omega))$. It was extended and generalized in several papers 
by Ambrosio, Crippa,  De Lellis \cite{AmbrC}, \cite{CDL} and others. Further deep generalization of this approach consisting in replacing
the condition imposed on the divergence of $\vu$ by a weaker condition postulating that "$\vu$ is weakly incompressible" is due to Bianchini, Bonicatto
\cite{BIBO}. The latter result (which is essentially about the properties of the flow  of the vector field $\vu$) implies as a corollary the uniqueness for the pure 
transport equation under "weak incompressibility" condition. (It is not without interest, that a stronger form of this corollary can be obtained
within the Sobolev functional setting quite easily by the purely Eulerian approach \cite[Proposition 5]{AN}.) In contrast with conservation laws, where the
 $BV(\Omega)$ theory found many applications, it { has} not so far appeared to be exploitable in the theory of compressible { Navier--Stokes}
equations.

{ The paper is organised as follows. In Section \ref{S2} we introduce various notions of solutions to the continuity and transport equations that will be used in the sequel. Section \ref{M} is devoted to the formulation of the main results, and of the auxiliary results needed for their proofs, which are of independent interest.
Theorems \ref{t2} and \ref{t3} (and Corollary \ref{c1}  in Subsection \ref{M1}) deal with the properties of vacuum  in any renormalized time integrated weak solution of the continuity equation. This implies immediately the same properties of vacuum in any renormalized weak solution to the compressible Navier--Stokes equations. This issue is discussed in Subsection \ref{M3} (see namely Corollary \ref{c1} and Remark \ref{r1}). Theorems \ref{t2}--\ref{t3} and Corollaries \ref{cor1}, \ref{cor2} and \ref{c1}
are main results of the paper. Their proofs require a good understanding
of the relation between various types of solutions introduced in Section \ref{S2}. This issue of independent interest is treated in Subsection \ref{M2}.
The { matters} of time integration of renormalized distributional of weak solutions are treated in Subsection \ref{M2.1} (see Theorems \ref{th1}, \ref{th4}). 
The passage from distributional to  renormalized weak solution is handled in Subsection \ref{M2.2}  (see Theorems \ref{th3}, \ref{th5}). The passage from continuity and pure transport equation to a continuity equation is  formulated in Subsection \ref{M2.3} (see Theorem \ref{th5}). The remaining part of the paper is devoted
to the proof of Theorems (\ref{t2}--\ref{th5}). Section \ref{S4} collects three preliminary classical results whose conclusions will be frequently used throughout the proofs. Section \ref{S5} is devoted to the proof of Theorems \ref{th1}--\ref{th2}, Section \ref{S6} to the proof of Theorems \ref{th3}--\ref{th4} and Section \ref{S7}
to the proof of Theorem \ref{th5}. Finally in the last Section we combine the results of Theorems \ref{th1}--\ref{th5} to prove the main theorems: Theorems \ref{t2} and \ref{t3}.} 

We finish this  section by introducing the functional spaces and some notations.
In what follows, we use standard notation for the Lebesgue and Sobolev spaces ($L^p(\Omega)$ and $W^{1,p}(\Omega)$ with the corresponding norms  $\|u\|_{L^p(\Omega)}$ and $\|u\|_{W^{1,p}(\Omega)}$, respectively). We do not distinguish the notation for the norms for scalar- and vector-valued functions. However, the vector-valued functions are prin\-ted boldface and we write $\vu \in L^p(\Omega;R^d)$ instead of $\vu \in L^p(\Omega)$, similarly for other functions spaces. For function spaces of time and space dependent function we use the standard notation for the Bochner spaces $L^p(I;L^q(\Omega))$ or $L^p(I;L^q(\Omega;R^d))$, respectively. { We also use the notation $C([0,T];L^p(\Omega))$ for continuous functions on interval [0,T] with values in $L^p(\Omega)$ and $C_{{\rm weak}}([0,T];L^p(\Omega))$ a vector subspace of $L^\infty(0,T;L^p(\Omega))$ of functions continuous on $[0,T]$ with respect to the weak topology
of { $L^p(\Omega)$.} More exactly, a function $f:[0,T]\mapsto L^p(\Omega)$ (defined on $[0,T]$) belongs to $C_{{\rm weak}}([0,T];L^p(\Omega))$ iff
$f\in L^\infty(0,T;L^p(\Omega))$ and for all $\eta\in L^{p'}(\Omega)$ the map $\tau\mapsto \intO{f(\tau)\eta}$ is continuous on interval $[0,T]$.
For the norms in Bochner spaces we use the function space as full index, as e.g. $\|u\|_{L^{p}(I;L^q(\Omega))}$ or { $\|u\|_{L^p(I;W^{1,q}(\Omega))}$}.  Throughout the paper, the constants are denoted by $C$ and their value may change even in the same formula.}

\section{Various notions of solutions to continuity and pure transport equations} \label{S2}

{ The main results of this paper will largely rely on various notions of (weak) solutions to the
continuity and pure transport equations and their inter-relations. We shall introduce these notions in this section.}

We consider the equations on the time-space cylinder $Q=I\times\Omega$, $\Omega$ a bounded open set in $R^d$, $d\ge 2$, 
and $I=(0,T)$, $T>0$ a time interval. The equations read:
\begin{enumerate}
\item {Continuity equation}
\begin{equation}\label{co1}
{\partial_t\vr+{\rm div}\,(\vr\vu)=0\;\mbox{in $(0,T)\times\Omega$}}
\end{equation}
with initial condition
$$
\vr(0,\cdot)=\vr_0(\cdot)\;\mbox{in $\Omega$}.
$$
\item {Pure transport equation}
\begin{equation}\label{tr1}
{ \partial_t s+\vu\cdot\Grad s=0\;\mbox{in $(0,T)\times\Omega$}}
\end{equation}
with initial condition
$$
s(0,\cdot)=s_0(\cdot)\;\mbox{in $\Omega$}.
$$
\end{enumerate}

We shall consider several different notions of solutions to these equations. 

\begin{df}[Continuity equation]\label{dfco} 
Let 
\begin{equation}\label{covu}
{\vu\in L^{1}(I\times\Omega;R^d)},\;{\rm div}\,  \vu\in L^1(I\times\Omega).
\end{equation}
We say that function 
\begin{equation}\label{cocl1}
\vr\in L^1(I\times\Omega)\;\mbox{such that $\vr\vu\in  L^1(I\times\Omega;R^d)$}
\end{equation}
is\footnote{
In some cases, it would be enough to assume { $\vu\in L_{\rm loc}^{1}(I\times\Omega;R^d)$}, ${\rm div}\,\vu\in L_{\rm loc}^1(I\times\Omega)$ and
condition (\ref{cocl1}) could be weaken to $\vr\in L_{\rm loc}^1(I\times\Omega)$, such that $\vr\vu\in L_{\rm loc}^1(I\times\Omega;R^d)$.
We do not consider this situation since it is irrelevant from the point of view of the present paper.
}:
\begin{enumerate}
\item
{\em Distributional solution} to the continuity equation (\ref{co1})
iff it satisfies (\ref{co1}) in the sense of distributions over the time-space, namely iff
\begin{equation}\label{co1.1}
\int_0^T \int_\Omega (\vr \partial_t \varphi + \vr \vu\cdot \nabla \varphi) \dxdt = 0
\end{equation}
holds for arbitrary  $\varphi \in C^\infty_{\rm c}(I\times \Omega)$. 
\item {\em Weak solution} to the continuity equation (\ref{co1}) iff
\begin{equation}\label{co1.2}
\mbox{equation (\ref{co1.1}) holds with arbitrary  $\varphi \in C^\infty_{\rm c}(I\times \overline\Omega)$.}
\end{equation} 
\item {\em Time integrated distributional solution} to the  continuity equation (\ref{co1}) iff $\vr \in C_{\rm weak}(\overline I;L^1(\Omega))$
 and  it holds
\begin{equation}\label{co1.3}
\int_\Omega(\vr\varphi)(\tau,\cdot)  \dx - \int_\Omega (\vr\varphi) (0,\cdot) \dx =
\int_0^\tau \int_\Omega (\vr \partial_t \varphi + \vr \vu\cdot \nabla \varphi) \dxdt 
\end{equation} 
for any $\varphi \in C^\infty_c(\overline I\times {\Omega})$ and any $\tau \in\overline I$.
\item {\em Time integrated weak solution} to the continuity equation (\ref{co1}) iff $\vr \in C_{\rm weak}(\overline I;L^1(\Omega))$
 and 
\begin{equation}\label{co1.4}
\mbox{equation (\ref{co1.3}) holds with arbitrary  $\varphi \in C^\infty_{\rm c}(\overline I\times \overline\Omega)$ and any $\tau \in\overline I$.}
\end{equation} 
\item {\em Renormalized distributional solution} to the  continuity equation (\ref{co1}) iff in addition to (\ref{co1.1}), 
\begin{equation}\label{rco1.1}
\int_0^T \int_\Omega \Big(b(\vr) \partial_t \varphi + b(\vr) \vu\cdot \nabla \varphi - \big(b'(\vr)\vr -b(\vr)\big) \Div \vu \varphi\Big) \dxdt = 0
\end{equation}
holds with all $\varphi \in C^\infty_{\rm c}(I\times \Omega)$ and all renormalizing functions\footnote{Conditions (\ref{ren}), (\ref{covu}) and (\ref{cocl1}) immediately ensure that $b(\vr)$, $b(\vr)\vu$ and $(\vr b'(\vr)-b(\vr)){\rm div}\,\vu\in L^1(I\times\Omega)$. As will be seen later, in fact $b(\vr) \in C([0,T];L^1(\Omega))$, too.}
\begin{equation}\label{ren}
b \in C^1([0,\infty)),\; b'\in C_c([0,\infty)).
\end{equation}
\item {\em Renormalized weak solution} to the continuity equation (\ref{co1}) iff  in addition to (\ref{co1.2}), 
\begin{equation}\label{rco1.2}
\mbox{equation (\ref{rco1.1}) holds with all $\varphi \in C^\infty_{\rm c}(I\times \overline\Omega)$ and all $b$ in (\ref{ren}).}
\end{equation}
\item {\em Renormalized time integrated distributional solution} to the continuity equation (\ref{co1}) iff
$b(\vr) \in C_{\rm weak}(\overline I;L^1(\Omega))$ and 
 in addition to (\ref{co1.3}), 
\begin{equation}\label{rco1.3}
\int_\Omega (b(\vr)\varphi) (\tau,\cdot) \dx - \int_\Omega (b(\vr)\varphi)(0,\cdot) \dx=
\end{equation}
$$
\int_0^T \int_\Omega \Big(b(\vr) \partial_t \varphi + b(\vr) \vu\cdot \nabla \varphi - \big(b'(\vr)\vr -b(\vr)\big) \Div \vu \varphi\Big) \dxdt
$$
holds with all $\varphi \in C^\infty_{\rm c}(\overline{I}\times \Omega)$, all $\tau\in\overline I$ and all renormalizing functions $b$ in the class (\ref{ren}).
\item {\em Renormalized time integrated weak solution} to the continuity equation (\ref{co1}) iff $b(\vr) \in C_{\rm weak}(\overline I;L^1(\Omega))$and in addition to (\ref{co1.4}), 
\begin{equation}\label{rco1.4}
\begin{aligned}
\mbox{equation (\ref{rco1.3}) holds with all $\varphi \in C^\infty_{\rm c}(\overline I\times \overline\Omega)$,} \\ \mbox{all $\tau\in\overline I$ 
and all $b$ in (\ref{ren}).}
\end{aligned}
\end{equation}
\end{enumerate}
\end{df}

{
Due to the presence of term containing $s{\rm div}\,\vu$ in the weak formulation of the pure transport equation,
the definition of weak solutions/renormalized weak solutions in this case asks for better summability of
the quantity $s$ (compared to the summability required for $\vr$ expressed through assumption (\ref{cocl1}) in the case of the continuity equation). 


\begin{df}[Pure transport equation]\label{dftr}
Let $\vu$ satisfy (\ref{covu}).
We say that function 
\begin{equation}\label{trcl1}
s\in L^1(I\times\Omega)\;\mbox{such that $s\vu$ and $s{\rm div}\,\vu\in  L^1(I\times\Omega)$}
\end{equation}
is\footnote{
In some cases, it would be enough to assume { $\vu\in L_{\rm loc}^{1}(I\times\Omega;R^d)$}, ${\rm div}\,\vu\in L_{\rm loc}^1(I\times\Omega)$ and
condition (\ref{trcl1}) could be weaken to $s\in L_{\rm loc}^1(I\times\Omega)$, such that $s\vu, s {\rm div}\,\vu\in L_{\rm loc}^1(I\times\Omega)$.
We do not consider this situation since it is irrelevant from the point of view of the present paper.
}:
\begin{enumerate}
\item
{\em Distributional solution} to the pure transport equation (\ref{tr1})
iff it satisfies (\ref{tr1}) in the sense of distributions over the time-space, namely iff
\begin{equation}\label{tr1.1}
\int_0^T \int_\Omega (s \partial_t \varphi + s \vu\cdot \nabla \varphi+ s{\rm div}\,\vu \varphi) \dxdt = 0
\end{equation}
holds for arbitrary  $\varphi \in C^\infty_{\rm c}(I\times \Omega)$. 
\item {\em Weak solution} to the pure transport equation (\ref{tr1}) iff
\begin{equation}\label{tr1.2}
\mbox{equation (\ref{tr1.1}) holds with arbitrary  $\varphi \in C^\infty_{\rm c}(I\times \overline\Omega)$.}
\end{equation} 
\item {\em Time integrated distributional solution} to the pure transport equation (\ref{tr1}) iff $s \in C_{\rm weak}(\overline I;L^1(\Omega))$
 and  it holds
\begin{equation}\label{tr1.3}
\int_\Omega(s\varphi)(\tau,\cdot)  \dx - \int_\Omega (s\varphi) (0,\cdot) \dx =
\int_0^\tau \int_\Omega (s \partial_t \varphi + s \vu\cdot \nabla \varphi) \dxdt 
\end{equation} 
for any $\varphi \in C^\infty_c(\overline I\times {\Omega})$ and any $\tau \in\overline I$.
\item {\em Time integrated weak solution} to the pure transport equation (\ref{tr1}) iff $s \in C_{\rm weak}(\overline I;L^1(\Omega))$
 and 
\begin{equation}\label{tr1.4}
\mbox{equation (\ref{tr1.3}) holds with arbitrary  $\varphi \in C^\infty_{\rm c}(\overline I\times \overline\Omega)$ and any $\tau \in\overline I$.}
\end{equation} 
\item {\em Renormalized distributional solution} to the pure transport equation (\ref{tr1}) iff in addition to (\ref{tr1.1}), 
\begin{equation}\label{rtr1.1}
\int_0^T \int_\Omega \Big(b(s) \partial_t \varphi + b(s) \vu\cdot \nabla \varphi + b(s) \Div \vu \varphi\Big) \dxdt = 0
\end{equation}
holds with all $\varphi \in C^\infty_{\rm c}(I\times \Omega)$ and all renormalizing functions $b$ belonging to class (\ref{ren}).
\item {\em Renormalized weak solution} to the pure transport equation (\ref{tr1}) iff  in addition to (\ref{tr1.2}), 
\begin{equation}\label{rtr1.2}
\mbox{equation (\ref{rtr1.1}) holds with all $\varphi \in C^\infty_{\rm c}(I\times \overline\Omega)$ and all $b$ in (\ref{ren}).}
\end{equation}
\item {\em Renormalized time integrated distributional solution} to the pure transport equation (\ref{tr1}) iff
$b(\vr) \in C_{\rm weak}(\overline I;L^1(\Omega))$ and 
 in addition to (\ref{tr1.3}), 
\begin{equation}\label{rtr1.3}
\int_\Omega (b(s)\varphi) (\tau,\cdot) \dx - \int_\Omega (b(s)\varphi)(0,\cdot) \dx=
\end{equation}
$$
\int_0^T \int_\Omega \Big(b(s) \partial_t \varphi + b(s) \vu\cdot \nabla \varphi + b(s) \Div \vu \varphi\Big) \dxdt 
$$
holds with all $\varphi \in C^\infty_{\rm c}(\overline{I}\times \Omega)$, all $\tau\in\overline I$ and all renormalizing functions $b$ 
in the class (\ref{ren}).
\item {\em Renormalized time integrated weak solution} to the pure transport equation (\ref{tr1}) iff $b(s) \in C_{\rm weak}(\overline I;L^1(\Omega))$ and in addition to (\ref{co1.4}), 
\begin{equation}\label{rtr1.4}
\begin{aligned}
\mbox{equation (\ref{rtr1.3}) holds with all $\varphi \in C^\infty_{\rm c}(\overline I\times \overline\Omega)$,} \\
\mbox{all $\tau\in\overline I$ 
and all $b$ in (\ref{ren}).}
\end{aligned}
\end{equation}
\end{enumerate}
\end{df}
}

\section{Main results}\label{M}

The primal goal of this paper is the investigation of the vacuum formation in the weak solution (density, velocity)---$(\vr,\vu)$---in the compressible Navier--Stokes equations. We shall prove that the volume of eventual vacuum set evolutes continuously in time and, more surprisingly, if there is 
no vacuum  at time $0$ and there is a vacuum of non-zero measure at some time $\tau\in (0,T)$, 
then any distributional solution $R$ (with certain reasonable summability properties) to
the continuity equation (with the same transporting velocity $\vu$)---if it exists---admits at time $\tau$
a larger vacuum set $\{x\in\Omega|R(\tau)=0\}$ than the vacuum set of $\vr$. This property does not imply absence of vacuum but pleads in favour of the 
sparseness of the event
of creation of vacuum in compressible flows.

All these properties rely exclusively on the properties of continuity and transport equations. We shall therefore
formulate them as such in Subsection \ref{M1}, postponing the formulation in the context of Navier--Stokes equations 
to Subsection \ref{M3}.

The proofs of results in Subsection \ref{M1} rely {essentially} on the properties and inter-relation of various types of weak/renormalized solutions
to the continuity and transport equations and their combinations, which are of independent interest.  Bits of pieces of some of these results 
(all of them having ground in the seminal work by DiPerna and Lions \cite{DL}) 
are non systematically spread through the mathematical literature in several (mostly recent)  
papers dealing with the existence of weak solutions to the compressible Navier--Stokes equations and compressible mixtures as auxiliary tools,
\cite{FeBook}, \cite{FeNo_book}, \cite{FeNoPe}, \cite{NoPo}, \cite{NoSt}, \cite{VWY}. We will state in Subsection \ref{M2} those of these properties needed
in this paper in their full generality and provide their detailed proofs.

\subsection{Properties of vacuum in the weak solution of the continuity equation} \label{M1}

The first theorem dealing with vacuum sets in the continuity equation reads.

\begin{thm} \label{t2} Let $\Omega\subset R^d$ be a bounded domain.
Let $1\le q,p\le \infty$ and $\vu \in L^{p}(0,T; W^{1,q}(\Omega;R^d))$. Let 
\begin{equation}\label{gamma+}
 0\le\vr \in C_{{\rm weak}}(\overline I;L^\gamma(\Omega)),\; \gamma>1
\end{equation}
be a renormalized time integrated  weak solution to the continuity equation (\ref{co1}) with
transporting velocity $\vu$ (i.e. it belongs to class (\ref{cocl1}), satisfies equation (\ref{co1.4}) and  equation (\ref{rco1.4}) with the renormalizing functions 
$b$ from (\ref{ren})).

  Then the map $t\mapsto s_\vr(t,\cdot):= { 1_{\{x\in \Omega|\vr(t,x)=0\}}(\cdot)}$
belongs to $C([0,T];L^r(\Omega))$ with any $1\le r<\infty$ and it is a time integrated renormalized weak solution of the pure transport equation (\ref{tr1}) with transporting velocity $\vu$. In particular,
\begin{equation} \label{2.8}
{ |\{x\in \Omega| \vr(t,x) = 0\}|_d \in C([0,T]).}
\end{equation}
In the above $|A|_d$ denotes the $d$-dimensional Lebesgue measure of the set A.
\end{thm}

The second theorem about the vacuum issue reads.

\begin{thm} \label{t3}
Let $\Omega\subset R^d$ be a bounded Lipschitz domain.
Let
\begin{equation}\label{qpab}
1\le q,p,\alpha,\beta\le \infty,\; (q,\beta)\neq (1,\infty),\;\frac 1\beta+\frac 1q\le1,\;\frac 1\alpha+\frac 1p\le 1.
\end{equation}
Let $\vr$ from {class} (\ref{gamma+})
be a renormalized time integrated weak solution to the continuity equation (\ref{co1}) with
transporting velocity $\vu \in L^{p}(0,T; W^{1,q}_0(\Omega;R^d))$ (i.e. it belongs to class (\ref{cocl1}), satisfies equation (\ref{co1.1}) and
equation (\ref{rco1.1}) with renormalizing functions $b$ from (\ref{ren})).

Let
\begin{equation}\label{clR-} 
0\le R\in  L^\infty(0,T;L^{\widetilde\gamma}(\Omega))\cap L^\alpha(0,T; L^\beta(\Omega)),\; \widetilde\gamma>1
\end{equation}  
be a distributional solution  to the continuity equation (\ref{co1}) with the same transporting velocity $\vu$.

Then 
\begin{enumerate}
\item 
Function $R$ belongs to
\begin{equation}\label{clR}
R\in C_{\rm weak}([0,T];L^{\widetilde\gamma}(\Omega))\cap C([0,T]; L^r(\Omega)),\; 1\le r<\widetilde\gamma
\end{equation}
and it
is a renormalized time integrated weak solution of the continuity equation (\ref{co1}).
\item {The map $t\mapsto (s_\vr R)(t)$} belongs to $C([0,T];L^r(\Omega))$ with any $1\le r<\widetilde\gamma$ and it is a renormalized time integrated weak solution
of the continuity equation (\ref{co1}) (with the same transporting velocity). In particular,
\begin{equation} \label{clRR}
{\int_{\Omega} (s_\vr R)(t,\cdot) \dx = \int_{\Omega} (s_\vr R)(0,\cdot) \dx}
\end{equation}
for all $t \in [0,T]$. 
\item If further $\vr(0,\cdot)>0$ a.e. in $\Omega$, then, up to sets of $d$-dimensional Lebesgue measure zero, for all $t\in (0,T]$
$$
\{x\in \Omega| \vr(t,x) =0\} \subset \{x\in \Omega| R(t,x) =0\}.
$$
\end{enumerate}
\end{thm}

The second theorem has the following immediate consequences:

\begin{cor}\label{cor1}
Let $q$, $p$, $\alpha$, $\beta$ verify conditions (\ref{qpab}) and $\widetilde\gamma,\gamma>1$.
Let $\Omega$, $\vr$, $\vu$ verify assumptions of Theorem \ref{t3}, where $\vr(0,x)>0$. (In particular, $\vr$ is a renormalized time integrated weak  solution of the continuity equation 
(\ref{co1}) with transporting velocity $\vu$.)

Let $\tau\in (0,T)$. Suppose that continuity equation (\ref{co1})
 with transporting velocity  $\vu$
admits at least one  distributional solution $R$ 
 belonging to class  (\ref{clR-})
 which does not admit in $\Omega$ a vacuum at time $\tau$, i.e.
$R(\tau)>0$ a.e. in $\Omega$.

Then $\vr$ does not admit a vacuum at time $\tau$, i.e. 
$$
 |\{x\in \Omega| \vr(\tau,x) =0\}|_d=0. 
$$ 
\end{cor}

\begin{cor}\label{cor2}
Let $q,\alpha,\beta,\gamma,\widetilde\gamma$ verify assumptions of Corollary \ref{cor1} with $p=\infty$.

Let $\Omega$, $\vr$, $\vu$ verify assumptions of Corollary \ref{cor1}. (In particular, $0\le \vr$ is a renormalized time integrated weak  solution of the continuity equation (\ref{co1})
with transporting velocity $\vu$ and $\vr(0,x)>0$.) We assume that $\vu$ is time independent, i.e. $\vu=\vu(x)$, { $\vu\in W^{1,q}_0(\Omega;R^d)$}.

 Suppose that continuity equation (\ref{co1})
 with transporting velocity  $\vu$
admits at least one (local in time)  distributional solution $R$ on $(0,T')\times\Omega$ with some $T'>0$ 
 belonging to class (\ref{clR-})$_{T=T'}$ 
 which does not admit in $\Omega$ a vacuum at time $\tau\in (0,T')$, i.e. there exists $\tau\in (0,T')$ such that
$R(\tau)>0$ a.e. in $\Omega$.

Then $\vr$ does not admit a vacuum at any time in $[0,T]$, i.e. 
$$
\forall t\in [0,T],\; |\{x\in \Omega| \vr(t,x) =0\}|_d=0. 
$$ 
\end{cor}

\begin{rmk}\label{rmkd0}
\begin{enumerate}
\item
In practice, if $\vu \in L^{p}(0,T; W^{1,q}(\Omega;R^d)),$ condition (\ref{cocl1}) in Theorems \ref{t2} and \ref{t3} is ensured by assumption
\begin{equation}\label{gamma} 
1<\gamma\le \infty,\;\frac 1\gamma+\frac 1 q\le 1+\frac 1d. 
\end{equation}
Alternatively, condition (\ref{cocl1}) can be achieved by requiring  $\vu \in L^{p}(0,T;$ $ W^{1,q}(\Omega;R^d)),$ $\vr\in L^\alpha(0,T;L^\beta(\Omega))$,
where $p,q,\alpha,\beta$ verifies (\ref{qpab}). In the theory of weak solutions to compressible Navier--Stokes equations, the former setting provides 
stronger results, cf. Section \ref{M3}.
\item Condition $\vu|_{I\times\partial\Omega}=\vc{0}$ in  Theorem \ref{t3} and Corollaries \ref{cor1}, \ref{cor2} can be replaced by 
$\vu\cdot\vc n|_{I\times\partial\Omega}=0$.
\item  We notice that Theorem \ref{t2} holds independently of the boundary condition imposed on $\vu$ at the boundary (since it deals
with weak solutions in the sense of Definition \ref{dfco}. This is not the case of
Theorems \ref{t3} and Corollaries \ref{cor1}, \ref{cor2}. Nevertheless, they continue to hold
if we replace $W_0^{1,q}(\Omega)$ by $W^{1,q}(\Omega)$ provided we suppose that
$R$ is { a renormalized time integrated {\rm weak} solution (instead of  a renormalized time integrated {\rm distributional} solution).}
Anyway, however, in all these cases the condition $\vr\vu\cdot\vc n|_{I\times\partial\Omega}=0$ must always be satisfied
al least in the weak sense; it is implicitly required in the weak formulation of the equation through the fact that the test functions do not
vanish on the boundary.
\end{enumerate}
\end{rmk}

\subsection{Relations between various types of solutions to continuity and pure transport equations}\label{M2}

The proofs of Theorems \ref{t2} and \ref{t3} are based
on the systematic study of relations and properties of the various types of weak solutions to the continuity and pure transport equations
and their inter-relations. In this section we formulate the adequate results. They are, indeed, of independent interest.

\subsubsection{Time integration of renormalized distributional/weak solutions} \label{M2.1}

The main message of this subsection is the  observation that any renormalized distributional (or weak) solution of the continuity equation/pure
transport equation (introduced in Definitions \ref{dfco}--\ref{dftr}) admits---under certain reasonable conditions---a representative that 
is continuous on the time interval $[0,T]$ with values in
$L^1(\Omega)$, and that both continuity/pure transport and renormalized continuity/pure transport equations can be integrated up to the end-points 
of any time interval $[0,\tau]$, $\tau\in [0,T]$.

\begin{thm}[Continuity equation] \label{th1} 
Let $\Omega\subset R^d$, $d\ge 2$ be a bounded domain with Lipschitz boundary. Let { $\vu\in L^{p}(I; W^{1,q}(\Omega;R^d))$}, $1\le p,q\le \infty$.
Suppose that
\begin{equation}\label{t1.1}
0\le \vr\in L^\infty(I;L^\gamma(\Omega)),\;\gamma>1.
\end{equation}
Then the following statements are true:
\begin{enumerate}
\item If $\vr$ is a {\em renormalized distributional solution} of the continuity equation  with transporting velocity $\vc u$
(i.e. it belongs to class (\ref{cocl1}) and satisfies (\ref{co1.1}), (\ref{rco1.1}) with any renormalizing function $b$ from (\ref{ren})), then
function $\vr$ and functions $b(\vr)$ with any $b$ from (\ref{ren}) belong to the class (\ref{clR})$_{\widetilde\gamma=\gamma}$
 and $\vr$ is a {\em renormalized time integrated distributional solution} of the continuity equation  with transporting velocity $\vc u$ (i.e. it belongs to class (\ref{cocl1})  
and  satisfies  identities (\ref{co1.3}) and (\ref{rco1.3}) with any renormalizing function
$b$ from (\ref{ren})).
\item If $\vr$ is a {\em renormalized weak solution} of the continuity equation with transporting velocity $\vc u$ (i.e. it belongs to class (\ref{cocl1}) and it satisfies equations
(\ref{co1.2}), (\ref{rco1.2}) with any $b$ from (\ref{ren})), then function $\vr$ and functions $b(\vr)$ with any $b$ from (\ref{ren}) belong to the class
(\ref{clR})$_{\widetilde\gamma=\gamma}$ and  it is a
{\em renormalized time integrated  weak solution} of the continuity equation with transporting velocity $\vc u$ 
(i.e. it belongs to class (\ref{cocl1})  and  it satisfies identities (\ref{co1.4}) and (\ref{rco1.4}) with any renormalizing function
$b$ in class (\ref{ren})).
\item Particularly, in both cases, $\vr \in C(\overline{I};L^r(\Omega))$, $1\leq r<\gamma$.
\end{enumerate}
\end{thm}

The same statement holds for the pure transport equation.
The theorem reads:

\begin{thm}[Pure transport equation]\label{th2} 
Let $\Omega$ and $\vu$ satisfy assumptions of Theorem \ref{th1} and let $s$ fulfill (\ref{t1.1}). 
Then the following statements are true:
\begin{enumerate}
\item If $s$ is a {\em renormalized distributional solution} of the pure transport equation with transporting velocity $\vc u$
(i.e. it belongs to class (\ref{trcl1}) and satisfies identities (\ref{tr1.1}), (\ref{rtr1.1})),  then $s$ and $b(s)$ with any $b$ from  (\ref{ren})
belong to class (\ref{clR})$_{\widetilde\gamma=\gamma}$
and $s$ is a {\em time integrated renormalized distributional solution} of the pure transport equation (i.e. it belongs to class (\ref{trcl1}) 
and it satisfies identities 
(\ref{tr1.3}) and (\ref{rtr1.3}) with any renormalizing function $b$ from (\ref{ren})).
\item If $s$ is a {\em renormalized weak solution } of the continuity equation with transporting velocity $\vc u$ 
(i.e. it belongs to class (\ref{trcl1}) and it satisfies identities 
(\ref{tr1.2}) and (\ref{rtr1.2})), then
$s$ and $b(s)$ with any $b$ from (\ref{ren})
belong to class (\ref{clR})$_{\widetilde\gamma=\gamma}$
and $s$
 is
{\em renormalized time integrated weak solution } (i.e. it belongs to class (\ref{trcl1}) and it satisfies identities (\ref{tr1.4}) and (\ref{rtr1.4}) with any renormalizing function
$b$ from (\ref{ren})).
\item Particularly, in both cases, $s \in C(\overline{I};L^r(\Omega))$, $1\leq r<\gamma$.
\end{enumerate}
\end{thm}

\begin{rmk}\label{remd1}
\begin{enumerate}
\item Concerning the continuity equation:
In practice, if $\vu \in L^{p}(0,T; W^{1,q}(\Omega;R^d)),$ condition (\ref{cocl1}) in Theorem \ref{th1} can be ensured by assumption
(\ref{t1.1}) with $\gamma$ from (\ref{gamma}).
If it is so, then the class of admissible renormalizing functions in Theorem \ref{th1} can be extended from (\ref{ren})
to\footnote{Here and in the sequel the exponent $q'$ is the H\"older conjugate exponent for $q$, $q_*$ is the { Sobolev  exponent} for $q$ (and
$q_*'$ is the H\"older conjugate exponent for $q_*$).}
\begin{equation}\label{t1.3}
b\in C^1([0,\infty)),\; b(\vr)\le c(1+s^{\gamma/q_*'}),\; \vr b'(\vr)-b(\vr)\le c(1+\vr^{\gamma/q'}).
\end{equation}
This is the setting that allows to get the strongest results in applications to weak solutions to compressible fluids, see Subsection \ref{M3}.

Alternatively, condition (\ref{cocl1}) can be achieved by requiring  $\vu \in L^{p}(0,T;$ $ W^{1,q}(\Omega;R^d)),$ $\vr\in L^\alpha(0,T;L^\beta(\Omega))$,
where $p,q,\alpha,\beta$ verify (\ref{qpab}), as mentioned in Remark \ref{rmkd0}. In this case one can take the true condition (\ref{t1.1}) with 
any $\gamma>1$. Condition (\ref{qpab}) is however more restrictive than (\ref{gamma}) from the point of view of applications to compressible fluids.
This setting is merely used only at the level of approximations of underlying compressible systems during the process of construction of weak solutions. Note finally that part of the first two claims of Theorem \ref{th1} hold without the requirement that the solutions is renormalized; i.e., if $\vr$ is a distributional solution, then under the assumptions of this theorem it is a time integrated distributional solution, similarly in the case of weak solution. On the other hand, Item 3. requires that the solution is renormalized.

\item Concerning the transport equation:
In practice, if the transporting velocity $\vu \in L^{p}(0,T;$ $ W^{1,q}(\Omega;R^d)$ and $\vr\in L^\alpha(0,T; L^\beta(\Omega))$, it is  condition (\ref{qpab}) which
guarantees satisfaction 
of condition (\ref{trcl1}) in Theorem \ref{th2}. In this situation
the class of admissible renormalizing 
functions in Theorem \ref{th2} can be extended from (\ref{ren})
to
\begin{equation}\label{t1.3+}
b\in C^1([0,\infty)),\; b(s)\le c(1+ s^{\gamma/q'}).
\end{equation}
{Note further that part of the first two claims of Theorem \ref{th2} hold without the requirement that the solutions is renormalized; i.e., if $s$ is a distributional solution, then under the assumptions of this theorem it is a time integrated distributional solution, similarly in the case of weak solution. On the other hand, Item 3. requires that the solution is renormalized.} 

\item  It appears that condition (\ref{qpab}) coincides with the conditions in the assumptions in the Friedrichs commutator lemma (see Lemma \ref{l2} later) 
which is the basic tool in the passage from distributional solutions to renormalized distributional solutions. The same condition 
is needed  in the passage from distributional to weak solutions in order to allow the application of the Hardy inequality near the boundary, cf. Theorem \ref{th3} for both features. This makes of the setting (\ref{qpab}) an universal  setting convenient for general transport equations (including continuity and pure transport).
This setting in the context of general transport equations has been introduced and fully exploited 
in the seminal DiPerna--Lions' paper \cite{DL}. 
\end{enumerate} 
\end{rmk}

\subsubsection{Passage from distributional to renormalized weak solutions} \label{M2.2}

The main message of this section is the  observation that, under certain assumptions (which are, in general, slightly stronger
than assumptions in the previous section), any distributional solution (time integrated
distributional solution) of the continuity equation/pure
transport equation (introduced in Definitions \ref{dfco}--\ref{dftr}) is a renormalized weak solution. 


\begin{thm} [Continuity equation] \label{th3} 
Let $\Omega\subset R^d$, $d\ge 2$ be a bounded domain with Lipschitz boundary. Further, let { $\vu\in L^{p}(I; W^{1,q}(\Omega;R^d))$,}
$0\le \vr\in L^\alpha(I;L^\beta(\Omega))$, where $p,q,\alpha,\beta$ satisfy condition (\ref{qpab}).

\begin{enumerate}
\item Assume that $\vr$
is a {\em distributional solution} of the continuity equation  with transporting velocity $\vc u$
(i.e. it satisfies (\ref{co1.1})). 
Then the following statements are true:
\begin{description}
\item {1.1}
 $\vr$ is a {\rm renormalized distributional solution}, i.e. it satisfies, in addition to equation (\ref{co1.1}), also equation (\ref{rco1.1}) with any renormalizing 
function $b$ in class (\ref{ren}).
\item {1.2} If moreover 
\begin{equation}\label{w0}
\vc u\in L^{p}(I; W^{1,q}_0(\Omega;R^d)),
\end{equation} 
then
$\vr$ is a {\em renormalized weak solution} of the continuity equation, i.e. $\vr$ satisfies continuity equation (\ref{co1.2}) and
its renormalized counterpart (\ref{rco1.2}) with any renormalizing function $b$ belonging to class (\ref{ren}).
\end{description}
\item
Assume  that $\vr$ belongs to class 
\begin{equation}\label{gamma++}
\vr\in C_{\rm weak}(\overline I; L^\gamma(\Omega))\;\mbox{with some $\gamma>1$}
\end{equation}
and
is a {\em time integrated  distributional solution} of the continuity equation  with transporting velocity $\vc u$
(i.e. it satisfies (\ref{cocl1}) and (\ref{co1.3})). 
Then the following statements are true:
\begin{description}
\item{2.1} Function $\vr$ belongs to (\ref{clR})$_{\widetilde\gamma=\gamma}$ and functions $b(\vr)$ with any $b\in (\ref{t1.3})$ belong to class 
(\ref{clR})$_{\widetilde\gamma=q_*'}$.
Moreover, $\vr$
is a {\it renormalized  time integrated  distributional solution} and it satisfies
 equation (\ref{rco1.3}) with any renormalizing function $b$ belonging to
(\ref{t1.3}).
\item{2.2} If moreover $\vu$ has zero traces (i.e. $\vu$ satisfies { (\ref{w0})),} then $\vr$ is a {\rm renormalized  time integrated weak solution} of the continuity equation
and it satisfies equations (\ref{co1.4}) and (\ref{rco1.4}) with any renormalizing function $b$ belonging to class (\ref{t1.3}).
\end{description}
\end{enumerate}
%



\end{thm}
 
\begin{thm}\label{th4}
 Exactly the same statement---only with minor modifications---is va\-lid for the pure transport equation. The modifications are the following:
\begin{enumerate}
\item
In assumptions of Statement 1., equation  (\ref{co1.1}) must be replaced by (\ref{tr1.1}), and further:

In Statement 1.1, equation (\ref{rco1.1}) must be replaced by (\ref{rtr1.1}).  In Statement 1.2, equations
 (\ref{co1.3}), (\ref{rco1.3}) must be replaced (\ref{tr1.3}), (\ref{rtr1.3}) 
and condition (\ref{t1.3}) by (\ref{t1.3+}). 
\item
In assumptions of Statement 2., equations (\ref{cocl1}) and (\ref{co1.3}) must be replaced by (\ref{trcl1}) and (\ref{tr1.3})
 and condition (\ref{t1.3}) by (\ref{t1.3+}),
and further:

In Statement 2.1, equations   (\ref{rco1.3}) must be replaced by
(\ref{rtr1.3}). In Statement 2.2, equations (\ref{co1.4}), (\ref{rco1.4}) must be replaced
by (\ref{tr1.4}), (\ref{rtr1.4}) and relation (\ref{clR})$_{\widetilde\gamma=q_*'}$ must be replaced by
(\ref{clR})$_{\widetilde\gamma=q'}$.
\end{enumerate}
\end{thm}

\subsubsection{From pure transport equation to continuity equation} \label{M2.3}

\begin{thm}\label{th5}
Let $\Omega$ be a bounded domain with Lipschitz boundary\footnote{As a matter of fact, the assumptions is important only in case of weak solutions. The result dealing with distributional solutions holds for arbitrary domain $\Omega$.}.
Suppose that
$$
1\le q,p,\alpha_\vr,\beta_\vr,\alpha_s,\beta_s\le \infty,\; (q,\beta_\vr)\neq (1,\infty),\; (q,\beta_s)\neq (1,\infty),
$$
$$
 \frac 1{\alpha_\vr}+ \frac 1{\alpha_s} +\frac 1p\le 1,\;
\frac 1{r_\vr}+ \frac 1{r_s} +\frac 1q\le 1,
$$
where
$$
 r_\vr\left\{\begin{array}{c} 
\in [1,\infty)\,\mbox{if $q>1$ and $\beta_\vr=\infty$}\\
=\beta_\vr\,\mbox{otherwise}
\end{array}\right\}, \; 
 r_s\left\{\begin{array}{c} 
\in [1,\infty)\,\mbox{if $q>1$ and $\beta_s=\infty$}\\
=\beta_s\,\mbox{otherwise}
\end{array}\right\}.
$$
Let
$$
\vr\in L^{\alpha_\vr}(I;L^{\beta_{\vr}}(\Omega)),\; s\in L^{\alpha_s}(I;L^{\beta_{s}}(\Omega)),\;{ \vu\in L^p(I;W^{1,q}(\Omega;R^d)).}
$$
Then there holds:
\begin{enumerate}
\item
Assume additionally that 
$$
\frac 1 {t_\vr} + \frac{1}{t_s} + \frac{1}{p} \leq 1,
$$
where
$$
 t_\vr\left\{\begin{array}{c} 
\in [1,\infty)\,\mbox{if $p>1$ and $\alpha_\vr=\infty$}\\
=\alpha_\vr\,\mbox{otherwise}
\end{array}\right\},\; 
 t_s\left\{\begin{array}{c} 
\in [1,\infty)\,\mbox{if $p>1$ and $\alpha_s=\infty$}\\
=\alpha_s\,\mbox{otherwise}
\end{array}\right\}.
$$ 
If $\vr$ is a distributional (resp. weak) solution   of the continuity equation (\ref{co1})  and $s$ a distributional (resp. weak) solution of the pure transport equation (\ref{tr1}) with transporting velocity $\vu$,
then $\vr s$ is a renormalized distributional (resp. weak) solution of the continuity equation with the same transporting velocity $\vu$.
\item If $\vr\in C_{\rm weak}(\overline I;L^{\gamma_\vr}(\Omega))$ is a time integrated distributional (resp. weak) solution of the continuity equation (\ref{co1})  and $s\in
C_{\rm weak}(\overline I;L^{\gamma_s}(\Omega))$  a time integrated distributional (resp. weak) solution of the pure transport equation (\ref{tr1}) with transporting velocity $\vu$ (where $1<\gamma_\vr,\gamma_s\le \infty$, $\frac 1{\gamma_\vr}+\frac 1{\gamma_s}:=\frac 1\gamma<1$), 
then $\vr s\in C(\overline I;L^r(\Omega))$, $1\le r<\gamma$ is a renormalized distributional (resp. weak) solution of the continuity equation with the same transporting velocity $\vu$.
\end{enumerate}
\end{thm}

\subsection{Application to compressible Navier--Stokes equations}\label{M3}

For simplicity, let us first recall the  compressible Navier--Stokes equations in barotropic regime: 
 
\begin{equation} \label{2.9}
\begin{aligned}
\partial_t \vr + \Div (\vr \vu) &= 0 \\
 \partial_t (\vr \vu) + \Div (\vr\vu \otimes \vu) + \nabla p(\vr) & = {\rm div}\,\tn S(\Grad\vu)+\vr \vc{f}\\
\end{aligned}
\end{equation}
which we consider  in $(0,T)\times\Omega$, together with the initial conditions in $\Omega$
\begin{equation} \label{2.10}
\vr(0,\cdot) = \vr_0, \qquad (\vr\vu)(0,\cdot) = \vc{m}_0
\end{equation}
and so called no-slip boundary condition on $(0,T)\times \partial \Omega$
\begin{equation} \label{2.11}
\vu(t,x) =\vc{0}.
\end{equation}
The homogeneous boundary condition  (\ref{2.11}) can be replaced by Navier (slip) boundary conditions or by periodic
boundary conditions if $\Omega$ is a periodic cell. 

In the above, $\tn S$ is the viscous stress tensor, which reads
\begin{equation}\label{S}
\tn S(\Grad\vu)=\mu\Big(\Grad\vu+\Grad\vu^t-\frac 2d{\rm div}\,\vu \tn I\Big)+\lambda{\rm div}\,\vu \tn I.
\end{equation}
The viscosity coefficients are assumed to be constant: $\mu >0$ and $\lambda \ge 0$. 
Function $\vr\mapsto p(\vr)$ denotes the pressure. One supposes  that 
$$
p\in C^1([0,\infty)).
$$

The classical (or strong) solutions, in general, may not exist (we can prove their existence either if the data are smooth and  the time interval is sufficiently short or if the data are in some sense additionally sufficiently small). We therefore consider the weak solutions. They are defined as follows:

\begin{df}\label{defws}
Let $\vr_0 \in L^{\gamma}(\Omega)$, $0\le\vr_0\in L^\gamma(\Omega)$ a.e. in $\Omega$, $\gamma>1$, $r>1$, $(\vr\vu)(0,\cdot) = \vc{m}_0 \in L^1(\Omega;R^d)$ and $\vc{f} \in L^\infty((0,T)\times \Omega;R^d)$. A couple $(\vr,\vu)$ is a renormalized weak solution to the initial boundary value problem (\ref{2.9}--\ref{2.11}) iff:
\begin{enumerate}
\item The couple $(\vr,\vu)$ belongs to functional spaces
$$
0\le \vr \in C_{{\rm weak}}(\overline I;L^\gamma (\Omega)),\; \vu \in L^2(I;W^{1,2}_0(\Omega;R^d)),\;p(\vr)\in L^1(Q),
$$
$$
\vr\vu \in C_{{\rm weak}}(\overline I; L^{r}(\Omega;R^d)), \;\vr (\vu\otimes \vu), p(\vr) \in  L^1((0,T)\times \Omega; R^{d\times d}).
$$
\item $\vr$ is a time integrated renormalized weak solution to the continuity equation \eqref{2.9}$_1$ with transporting velocity $\vu$.
\item The couple $(\vr,\vu)$ verifies the momentum equation \eqref{2.9}$_2$ in the sense of distributions.
\end{enumerate}
\end{df}

If Navier or periodic conditions are considered, the functional spaces and test functions in the above definition must be accordingly modified, see
\cite{FeNo_book}, \cite{NoSt} or \cite{FNPdom}.


\begin{cor} \label{c1}
Let $\gamma$ verify condition (\ref{gamma}) with $q=2$ (in particular $\gamma\ge 6/5$ if $d=3$). Then the claims of  
Theorems \ref{t2} and \ref{t3} (and Corollaries \ref{cor1}, \ref{cor2}) hold for any renormalized weak solution to the compressible 
Navier--Stokes equations specified in Definition  \ref{defws}.   
\end{cor}

\begin{rmk}\label{r1} 
\begin{enumerate}
\item
Note that renormalized weak solutions to the Navier--Sto\-kes equations with the regularity properties stated above (and, additionally, fulfilling the energy inequality)  can be constructed  with any of no-slip, Navier (slip) or periodic boundary conditions provided $\gamma>d/2$ and
$$
p(0)=0,\;p'(\vr)\ge a_1\vr^\gamma-b,\;p(\vr)\le a_2\vr^\gamma+b,\;\mbox{with some}\; a_1,a_2,b>0,
$$ 
\cite{FeNoPe} (for monotone pressure), \cite{Fe2002} (for non monotone pressure) and sufficiently regular { domains,} and \cite{NoNo}, \cite{Kuku} or \cite{Poul}
for a generalization to Lipschitz domains.
\item 
The above condition for pressure allows pressure functions which are non monotone on a compact portion of $[0,\infty)$.
In the case of periodic boundary conditions and provided $\gamma\ge 9/5$,  { this condition can be generalized allowing pressure 
functions non monotone  up to infinity and, also, another generalization allows small anisotropic perturbations of the isotropic 
stress tensor (\ref{S}), see Bresch, Jabin \cite[Theorems 3.1 and 3.2]{BrJa}.}
{
\item {Theorems \ref{t2}, \ref{t3} and Corollaries \ref{cor1}, \ref{cor2} also apply  to a couple $(\vr,\vu)$, where $(\vr,\vu,\vt)$---(density, velocity, temperature)---is a weak solution of the full Navier--Stokes--Fourier system,} constructed (according to different definitions of weak solutions
under different physical assumptions on constitutive laws and transport coefficients)  either in Feireisl \cite[Definition 7.1 and Theorem 7.1]{FeBook} or in \cite[Theorem 3.1]{FeNo_book} or in \cite{FeNoEd}, \cite{FNS}.
\item Theorems \ref{t2}, \ref{t3} and Corollaries \ref{cor1}, \ref{cor2} do not, in general, directly apply to a couple $(\vr,\vu)$ of weak solutions of Navier-Stokes equations
with degenerate density dependent viscosities unless it cannot be guaranteed that $\vu$ belongs to a Sobolev space of type $L^p(I;W^{1,q}(\Omega;R^d))$. In  fact, in this situation, typically, $\nabla\vu$ belongs to a Lebesgue space weighted by a positive power of $\vr$ (cf. { Bresch}, Desjardins \cite{BrDes}, Mellet, Vasseur \cite{MeVa},
Vasseur, Yu \cite{VY},
Li, Xin \cite{LiXin} for non exhausting relevant references).
}
\end{enumerate}
\end{rmk}


\section{Basic preliminaries}\label{S4}

Let us mention some standard preliminary tools. We shall use several times the theorem on Lebesgue points in the following form. 
\begin{lem}\label{Lbg}
Let  $f\in L^1(0,T;L^\gamma(\Omega))$, $1\le\gamma<\infty$. Then there exists $N\subset (0,T)$ of zero Lebesgue measure such that
for all $\tau\in (0,T)\setminus N$,
{ $$
\begin{aligned}
\lim_{h\to 0+} \frac 1h\int_{\tau-h}^\tau \|f(t,\cdot)-f(\tau,\cdot)\|_{L^\gamma(\Omega)}{\rm d}t&\to 0, \\ 
\lim_{h\to 0+} \frac 1h\int_\tau^{\tau+h} \|f(t,\cdot)-f(\tau,\cdot)\|_{L^\gamma(\Omega)}{\rm d}t&\to 0.
\end{aligned}
$$
Moreover, if $f\in C_{\rm weak}([0,T];L^\gamma(\Omega))$, then for any $\eta \in L^{\gamma'}(\Omega)$
$$
 \forall \tau\in [0,T),\;\intO{f(\tau,\cdot)\eta}= \lim_{h\to 0+}\frac 1h\int_{\tau}^{\tau+h}\Big(\intO{f(t,\cdot)\eta}\Big){\rm d}t
$$
and
$$
\sup_{\tau\in [0,T]}\|f(\tau,\cdot)\|_{L^\gamma(\Omega)}\le \|f\|_{L^\infty(0,T;L^{\gamma'}(\Omega))}.
$$}
\end{lem}

 We shall also  frequently use mollifiers. For the sake of completeness, we recall the basic facts. 
We denote by $j$ a function on $R^{d}$, $d \geq 1$, satisfying the following  requirements:
    $
    j \in C^{\infty}_c(R^{d}),\; {\rm supp}(j)=B(0,1),
    $
    $
    j(x)=j(-x),\; j \geq 0\;\mbox{on  $R^{d}$},\; \int_{R^d} j(x) {\rm d}x=1.
    $
    Next, for $\epsilon>0$, we denote by $j_\epsilon$ the  function
    $j_{\epsilon}(x):=\frac{1}{\epsilon^{d}}j(\frac{x}{\epsilon})$.
    For a given function $f\in L_{\rm loc}^1(R^{d})$, we finally define mollified $f$ as follows:
$
[f]_{\epsilon}:=f*j_{\epsilon} (x)=\int_{R^d}j_{\epsilon}(x-y) f(y){\rm d}y
$.
		
Let us recall the classical properties of these approximations.
    \begin{lem}\label{prop1}
		\begin{enumerate}
		\item
    If $1 \leq p < \infty$, then for any $f \in L^{p}(R^{d})$
    $$
		 [f]_\ep \in C^{\infty}(R^{d}) \cap {L^{p}}(R^{d}),\;
    \|{[f]_{\epsilon}}\|_{L^{p}(R^d)} \leq \|{f}\|_{L^{p}(R^d)}
		$$
		and
		$$
    f_\ep \rightarrow f\;\mbox{ in $L^p(R^d)$.}
    $$
		\item
    If $p=\infty$, then
		$$
    [f]_\ep \in C^{\infty}(R^{d}) \cap {L^{\infty}}(R^{d}),\;
    \|{[f]_{\epsilon}}\|_{L^{\infty}(R^d)} \leq \|{f}\|_{L^{\infty}(R^d)}.
		$$
		Moreover, if $f$ is uniformly continuous on { $R^d$,} then
		$$
		[f]_\ep\to f\;\mbox{in $C_b(R^d)$}.
		$$
		\item Let $1\le p\le\infty$. For all $f\in L^p(R^d)$, $g\in L^{p'}(R^d)$,
		$$
		\int_{R^d}[f]_\ep g \dx=\int_{R^d} f [g]_\ep \dx.
		$$
		\end{enumerate}
    \end{lem}
		
	The next lemma is the well-known Friedrichs lemma on commutators. It deals with the regularization of the
		quantity $\vu\cdot\Grad f$ 		defined in the sense of distributions as
		$$
		\vu\cdot\Grad f:={\rm div}(f\vu)-f{\rm div}\vu.
		$$
The lemma reads.
\begin{lem}[Friedrichs commutator lemma] \label{l2}
Let $I\subset R$ be an open bounded interval and $f \in L^\alpha(I;L^\beta_{{\rm loc}}(R^d))$, $\vu \in L^p(I;W^{1,q}_{{\rm loc}}(R^d;R^d))$. Let $1\leq q,\beta\leq \infty$, $(q,\beta)  \neq (1,\infty)$, $\frac 1q + \frac 1{\beta} \leq 1$, $1\leq \alpha \leq \infty$ and $\frac 1\alpha + \frac 1p \leq 1$.  Then
$$
[\vu \cdot \nabla f]_\ep - \vu \cdot \nabla [f]_\ep \to 0
$$
strongly in $L^t(I;L^r_{{\rm loc}} (R^d))$, where 
$$
\frac 1t \geq \frac 1\alpha + \frac 1p,\;  t\in [1,\infty)
$$
and 
$$
r\in [1,q) \text{ for } \beta = \infty, \ q \in (1,\infty],
$$ 
while $\frac 1\beta + \frac 1q \leq \frac 1r \leq 1$ otherwise. 
\end{lem}

\section{Proof of Theorems \ref{th1}--\ref{th2}}\label{S5}

The proof of Theorems { \ref{th1}--\ref{th2} is based} on the following two lemmas. The first lemma deals with distributional (or weak) solutions to conservation laws 
(\ref{dd}) and claims that their solutions admit, under certain conditions, $C_{\rm weak}([0,T],L^1(\Omega))$-representatives, and can be therefore
integrated up to the endpoints of any time interval $[0,\tau]\subset [0,T]$. 

\begin{lem}\label{aux3} 
Let $d \in L^{\infty}(I,L^{\gamma}(\Omega))$, $\gamma >1$ and $\vc{F} \in L^{1}(Q;R^{d})$, $G \in L^{1}(Q) $. 
\begin{enumerate}
\item Suppose that
\begin{equation}\label{dd}
\partial_{t}d+\Div \vc{F}+G=0 \quad \mbox{ in} \quad D'(Q).
\end{equation}
Then there exists a representative of $d$ such that it belongs to the space $C_{{\rm weak}}([0,T],L^{\gamma}(\Omega))$ and equation (\ref{dd})  can be integrated up to any time $\tau\in (0,T]$, i.e.
 $\forall \xi \in C^{1}([0,T])$, $\forall \tau \in (0,T]$ and $\forall \eta \in C_{c}^{1}(\Omega)$, there holds
\begin{equation}\label{ddi}
\int_{\Omega}d(\tau, x)\xi(\tau)\eta (x)\dx-\int_{\Omega}d(0,x)\xi(0)\eta(x)\dx=
\end{equation}
\[\int_{0}^{\tau}\int_{\Omega}\Big(d(t,x)\partial_{t}\xi(t)+\vc{F}(t,x)\cdot\nabla \eta(x)\xi(t) -G(t,x)\xi(t)\eta(x)\Big)\dxdt.\]
\item Suppose that (\ref{dd}) holds up to the boundary, i.e.
$$
\int_Q\Big(d\partial_t\varphi +\vc{F}\cdot\Grad\varphi-G\varphi\Big) \dxdt=0\;\mbox{for all $\varphi\in C^1_c((0,T)\times\overline\Omega)$.}
$$
Then there exists a representative of $d$ such that it belongs to the space $C_{{\rm weak}}([0,T],L^{\gamma}(\Omega))$ and equation (\ref{ddi}) holds $\forall \xi \in C^{1}([0,T])$, $\forall \tau \in (0,T]$ 
and $\forall \eta \in C_{c}^{1}(\overline\Omega)$.
\end{enumerate}
\end{lem}

\begin{proof}
We shall show only Statement 1. of Lemma \ref{aux3}. Statement 2. can be obtained repeating word by word the proof of Statement
1. with minor modifications.

We take in equation (\ref{dd}) test functions $\varphi(t,x)=\psi(t)\eta(x)$, where $\eta\in C^1_c(\Omega)$, and
\begin{equation*}
\psi(t) =\psi_{\tau,h}^+= 
 \begin{cases}
   \frac{1}{h}t  &\text{if $t\in [0,h] $}\\
    1 &\text{if $t\in [h,\tau] $}\\
    1-\frac{t-\tau}{h} &\text{if $t \in [\tau,\tau+h] $}\\
    0 &\text{if $t \in (\tau+h, +\infty). $}
 \end{cases}
\end{equation*}
Under assumptions on $d$, $\vc{F}$ and $G$ it is a folklore to show that this is an admissible test function in equation (\ref{dd}).

We obtain by direct calculation,
\begin{equation}\label{ddd}
\frac{1}{h}\int_{\tau}^{\tau+h}\int_{\Omega}d(t,x)\eta(x)\dxdt-\frac{1}{h}\int_{0}^{h}\int_{\Omega}d(t,x)\eta(x)\dxdt=
\end{equation}
\[\int_{0}^{\tau+h}\psi (t)\int_{\Omega}\vc{F}(t,x)\cdot \nabla \eta(x)\dxdt- \int_{0}^{\tau+h}\psi (t)\int_{\Omega}G(t,x)\eta(x)\dt \dx.\]
This identity leads to the following observations:
\begin{enumerate}
\item According to the theorem on Lebesgue points (cf. Lemma \ref{Lbg}), there is a set $N\subset(0,T)$ of zero Lebesgue measure $|N|=0$, such that
for all $\tau\in (0,T)\setminus N$, the limit $h\to 0+$ of the first expression exists. Since the limit
of the right hand side as $h\to 0+$ exists as well, we deduce that
$$
\forall\eta\in C^1_c(\Omega),\; 
\lim_{h\to 0+}\frac{1}{h}\int_{0}^{h}\int_{\Omega}d(t,x)\eta(x)\dxdt:= {\mathfrak{d}}_\eta(0+)\in { R.}
$$
The map $C^1_c(\Omega)\ni\eta\to {\mathfrak{d} }_\eta(0+)\in R$ is evidently {linear.} Moreover, since $d \in L^{\infty}(I,L^{\gamma}(\Omega))$,
we have estimate
$$
\sup_{0<h<T}\Big|\frac{1}{h}\int_{0}^{h}\int_{\Omega}d(t,x)\eta(x)\dxdt\Big|\le \|d\|_{L^\infty(0,T; L^\gamma(\Omega))}\|\eta\|_{L^{\gamma'}(\Omega))}
$$
by virtue of the H\"older inequality. In view of the Riesz representation theorem, we deduce that
there exists ${\mathfrak{d}}(0+)\in L^{\gamma}(\Omega)$ such that
$$
\forall\eta\in C^1_c(\Omega),\;
{\mathfrak{d}}_\eta(0+)=\int_\Omega {\mathfrak{d}}(0+)\eta\, {\rm d}x.
$$
\item Now, we take an arbitrary $\tau\in (0,T)$ and calculate limit $h\to 0+$ in equation (\ref{ddd}). We already know that for all $\eta\in C^1_c(\Omega)$ the limits 
of the second term at the left hand side and the limit of the right hand side exist and belong to $R$. We deduce from this fact that
$$
\lim_{h\to 0+}\frac{1}{h}\int_{\tau}^{\tau+h}\int_{\Omega}d(t,x)\eta(x)\dxdt:={\mathfrak{d}}_\eta(\tau+),
$$
where, by the same token as in the previous step,
$$
\forall\eta\in C^1_c(\Omega),\;
{\mathfrak{d}}_\eta(\tau+)=\int_\Omega{\mathfrak{d}}(\tau+)\eta\, {\rm d}x\;\mbox{with $\mathfrak{d}(\tau+)\in L^{\gamma}(\Omega)$}.
$$
\item
We test equation (\ref{dd})   by functions $\varphi(t,x)=\psi(t)\eta(x)$, where
\begin{equation*}
\psi(t) = \psi_{\tau,h}^-
 \begin{cases}
   \frac{1}{h}t  &\text{if $t\in [0,h] $}\\
    1 &\text{if $t\in [h,\tau] $}\\
    1-\frac{t-\tau+h }{h} &\text{if $t \in [\tau-h,\tau] $}\\
    0 &\text{if $t \in (\tau, +\infty). $}
 \end{cases}
\end{equation*}
It reads
\begin{equation}\label{ddd!}
\frac{1}{h}\int_{\tau-h}^{\tau}\int_{\Omega}d(t,x)\eta(x)\dxdt-\frac{1}{h}\int_{0}^{h}\int_{\Omega}d(t,x)\eta(x)\dxdt=
\end{equation}
\[\int_{0}^{\tau}\psi (t)\int_{\Omega}F(t,x) \nabla \eta(x)\dxdt- \int_{0}^{\tau}\psi (t)\int_{\Omega}G(t,x)\eta(x)\dt \dx.\]

\item By the same token as in Items 1. and 2. we define ${\mathfrak{d}}_\eta(\tau-)$ and $\mathfrak{d}(\tau-)\in L^{\gamma}(\Omega)$ for all
$\tau\in (0,T]$ . Subtracting (\ref{ddd}) and (\ref{ddd!}) and effectuating limit $h\to 0^+$, we obtain
$$
\forall \tau \in (0,T), \; \mathfrak{d}(\tau):={\mathfrak{d}(\tau+)}=  \mathfrak{d}(\tau-).
$$
We define
$$
{\mathfrak{d}}(0):=\mathfrak{d}(0+),\quad\mathfrak{d}(\tau):=\mathfrak {d}{(\tau+)},\;\tau\in (0,T),\quad \mathfrak{d}(T):=\mathfrak{d}(T-).
$$
We easily verify that $\mathfrak{d}$ satisfies equation (\ref{ddi}).

Subtracting (\ref{ddi}) with $\tau=\tau_1$ and $\tau=\tau_2$, $\tau_1,\tau_2\in [0,T]$ we readily verify that
{ $$
\forall\eta\in C^1_c(\Omega),\;\mbox{the map } \tau\mapsto\int_{\Omega}{\mathfrak{d}}(\tau)\eta\,{\rm d}x \mbox{ is continuous
on } [0,T].
$$}
Since $ C^1_c(\Omega)$ is dense in $L^{\gamma'}(\Omega)$, we finally conclude that
$$
{\mathfrak{d}}\in C_{\rm weak}([0,T]; L^\gamma(\Omega)).
$$  
\item According to theorem on Lebesgue points (cf. Proposition \ref{Lbg}), we have 
$$
d(\tau+)= d(\tau-)=d(\tau)=\mathfrak{d}(\tau)\;\mbox{a.e. in (0,T)}.
$$
This completes the proof of the fact that there exists a representative of $d$ such that $d\in C_{\rm weak}([0,T]; L^\gamma(\Omega))$.

\item It remains to show equation (\ref{ddi}). To this end we can repeat the whole procedure consisting of { Items 1.--5.}
with test functions $\varphi(t,x)=\psi(t)\xi(t)\eta(x)$, where $\psi=\psi^\pm_{\tau,h}$, $\xi\in C^1([0,T])$ and
$\eta\in C^1_c(\Omega)$.
\end{enumerate}
Lemma \ref{aux3} is thus proved.
\end{proof}

The continuity and pure transport equations are particular cases of equations investigated in Lemma \ref{aux3}. If we additionally know that
their solutions are renormalized, we can show that they not only belong to the class $C_{\rm weak}([0,T]; L^1(\Omega))$ but even to the class
 $C([0,T]; L^1(\Omega))$. 
{ This is subject of the second lemma.} 

\begin{lem}\label{aux4}
\begin{enumerate}
\item Let $\vu \in L^{p}(I,W^{1,q}(\Omega;R^d))$, $1\leq p\le \infty$, $1<q\leq \infty$,
\begin{equation}\label{cl1} 
\vr\in  L^{\infty}(I,L^{\gamma}(\Omega)),\;\gamma>1,\;\frac 1\gamma+\frac 1{q}\le 1+\frac 1d
\end{equation}
or
\begin{equation}\label{cl1+}
\vr \in L^{\infty}(I,L^{\gamma}(\Omega))\cap L^{p'}(I;L^{q'}(\Omega)), \gamma>1. 
\end{equation}
Suppose that $\vr$ is a renormalized distributional solution of the continuity equation (i.e. it satisfies (\ref{co1.1}), (\ref{rco1.1}) with 
renormalizing function $b$ in the class (\ref{ren})). Then there exists a representative of $\vr$ such that
$$
\vr \in C(\overline I;L^r(\Omega)),\; 1\le r<\gamma.
$$
\item The same statement, under the same assumptions on $\vu$ and
under assumption (\ref{cl1+}) holds for any renormalized distributional solution to the pure transport 
equation (satisfying (\ref{tr1.1}), (\ref{rtr1.1}) with 
renormalizing function $b$ in the class (\ref{ren})).
\end{enumerate}
\end{lem}

\begin{proof}
Again, it is enough to prove Statement 1. dealing with the continuity  equation. The proof of Statement 2. for the pure transport equation
requires only minor modifications and is, therefore, left to the reader as an exercise. It is to be noticed that, due to the presence of term $s{\rm div}\,\vu$ in the weak 
formulation of the pure transport equation, { Statement 2. is not true under assumption (\ref{cl1}) unless $\gamma \geq q'$. }

Employing Lemma \ref{aux3} (with $d=\vr$, $\vc{F}=\vr\vu$, $G=0$) we may suppose that $\vr\in C_{\rm weak}(\overline I;L^\gamma(\Omega))$.

Since $\vr$ is a renormalized distributive solution of the continuity equation, it satisfies
\begin{equation} \label{eq19}
\partial_{t}T_{k}(\vr)+\Div(T_{k}(\vr)\vu)+(\vr T_{k}'(\vr)-T_{k}(\vr))\Div \vu=0 \quad \mbox{in} \quad {\cal D}'(Q),
\end{equation}
where for any $k>1$
$$
T_{k}(\vr)=kT{\Big(\frac{\vr}{k}\Big)}\;\mbox{ with $T\in C^{1}([0,\infty))$}, 
$$
with
\begin{equation*}
T(s) = 
 \begin{cases}
    s &\text{if $0\leq s\leq 1 $}\\
    2 &\text{if $s \geq 3$.}\\
 \end{cases}
\end{equation*}

According to Lemma \ref{aux3} applied to (\ref{eq19}) with $d:=T_{k}(\vr)$, $\vc{F}:=T_{k}(\vr)\vu$ and $G:=(\vr T_{k}'(\vr)-T_{k}(\vr))\Div \vu$,
there exists
\begin{equation}\label{Tmol}
\T_{k}(\vr)\in C_{{\rm weak}}([0,T],L^{p}(\Omega)),\; \forall 1 \leq p < + \infty,
\end{equation}
 $$
(\mathcal{T}_{k}(\vr))(t)=T_{k}(\vr(t))\;\mbox{ a.a. in $\Omega$ for a.a. $t \in (0,T)$},
$$
 such that
\begin{equation} \label{eq19+}
\partial_{t}\mathcal{T}_{k}(\vr)+\Div(\mathcal{T}_{k}(\vr)\vu)+(\vr \mathcal{T}_{k}'(\vr)-\mathcal{T}_{k}(\vr))\Div \vu=0 \quad \mbox{in} \quad 
{\cal D}'(Q).
\end{equation}

 We can extend $\mathcal{T}_{k}(\vr)$ by $0$ outside $\Omega$ and regularize it  by using standard 
mollifiers over the space variables. The equation for mollified functions $[\T_{k}(\vr)]_{\ep}$ reads
 \begin{equation} \label{eq20}
 \partial_{t}[\T_{k}(\vr)]_{\ep}+\Div([\T_{k}(\vr)]_{\ep}\vu)+\Big[\big(\rho \T_{k}'(\vr)-\T_{k}(\vr)\big)\Div \vu\Big]_\ep=r_{\ep}
 \end{equation}
a.e. in 
$$
Q_\ep= I\times\Omega_\ep,\;\Omega_\ep=\{x\in \Omega\,|\,{\rm dist}(x,R^d\setminus\Omega)>\ep\},
$$
 where
\[
r_\ep:=r_{\ep}(\T_{k}(\vr),\vu)=\Div([\T_{k}(\vr)]_{\ep}\vu)-\Div[\T_{k}(\vr)\vu]_{\ep}\to 0\;\mbox{as $\ep\to 0$} 
\]
in  $L^{p}(I; L^{\widetilde q}( K))$, with any compact $K\subset\Omega$, $\widetilde q <q$ by virtue of  the Friedrichs lemma on commutators (cf. Lemma \ref{l2}).
  
Due to the standard properties of mollifiers 
$$
\Big[{ \vr} \T_{k}'(\vr)-\T_{k}(\vr)\Div \vu\Big]_\ep \to \vr \T_{k}'(\vr)-\T_{k}(\vr)\Div \vu
$$
in $L^{p}(I; L^{q}( K))$, $K\subset \Omega$, compact.
 
  On the other hand, since $\T_{k}(\vr)(t,\cdot))\in L^{r}(\Omega)$ {\em for all} $t \in [0,T], 1\leq r < + \infty$, we get by the same token, in
	particular,
 \begin{equation}\label{D0}
 \forall t \in [0,T] \quad [\T_{k}(\vr)(t, \cdot)]_{\ep} \to \T_{k}(\vr)(t, \cdot) \;\mbox{ in $L^{2}(K)$ with any compact $K\subset\Omega$.}
 \end{equation}

Moreover, since $\T_{k}(\vr)\in C_{\rm weak}([0,T]; L^p(\Omega))$, we infer that the mapping $t \mapsto [\T_{k}(\vr)]_{\ep}(\cdot, x)$ belongs
to $C([0,T])$ for all $x \in \Omega_\ep$ 
and hence $t \mapsto [\T_{k}(\vr)]_{\ep}^{2}(\cdot, x) \in C[0,T]$ for all $x \in \Omega_\ep$.
 %
%
Consequently,
 \[\Big(t \mapsto \int_{\Omega}[\T_{k}(\vr)]_{\ep}^{2}(t,x)\eta(x) \dx\Big) \in C([0,T])\]
for all $\eta \in C^1_c(\Omega)$ and $0<\ep<{\rm dist}({\rm supp}\,\eta,R^d\setminus\Omega)$. We deduce from estimate
\[
\sup_{t \in [0,T]}\int_{\Omega}[\T_{k}(\vr)]_{\ep}^{2}(t,x)\eta(x) \dx \leq  \sup_{t \in [0,T]} \|\T_{k}(\vr(t,\cdot))\|_{L^{2}(\Omega)}\|\eta\|_{L^{2}(\Omega)}
\]
\[\leq\|\T_{k}(\vr)\|_{L^\infty(0,T;L^{2}(\Omega))}\|\eta\|_{L^{2}(\Omega)}\leq C\]
 that the family of maps  
\begin{equation}\label{maps}
\Big\{t \mapsto \int\limits_{\Omega}[\T_{k}(\vr)]_{\ep}^{2}(t,x)\eta(x) \dx\,|\, 0<\ep<{\rm dist}({\rm supp}\,\eta, R^d\setminus\Omega)\Big\}
\end{equation}
is for any $k>1$ and any $\eta \in C^1_c(\Omega)$  equi-bounded in $C([0,T])$.
 
We multiply (\ref{eq20}) by $2[\T_{k}(\rho)]_{\ep}$, in order to get
 \begin{equation}\label{eqD1}
 \partial_{t}[\T_{k}(\vr)]_{\ep}^{2}+\Div([\T_{k}(\vr)]_{\ep}^{2}\vu)+[\T_{k}(\vr)]_{\ep}^{2}\Div \vu+
 \end{equation}
 \[ 
2[\T_{k}(\vr)]_{\ep}\Big[(\vr \T_{k}'(\vr)-\T_{k}(\vr))\Div \vu\Big]_\ep= 2[\T_{k}(\vr)]_{\ep}r_\ep \quad\mbox{ a.e. in $Q_\ep$}.
\]

Now, we take $\eta \in C^1_c(\Omega)$,
multiply equation (\ref{eqD1}) by $\eta$ and integrate over $\Omega$. We get, after an integration by parts,
$$
\partial_{t}\int_\Omega[\T_{k}(\vr)]_{\ep}^{2}\eta\, {\rm d} x-\int_\Omega[\T_{k}(\vr)]_{\ep}^{2}\vu\cdot\Grad\eta\, {\rm d}x+
\int_\Omega [\T_{k}(\vr)]_{\ep}^{2}\Div \vu\eta\, {\rm d}x+
$$
$$
\int_\Omega  2[\T_{k}(\vr)]_{\ep}\Big[\vr \T_{k}'(\vr)-\T_{k}(\vr)]\Div \vu\Big]_\ep\eta\, {\rm d}x= \int_\Omega 2[\T_{k}(\vr)]_{\ep}r_\ep\eta \dx ,
$$
where $0<\ep<{\rm dist}({\rm supp}\, \eta,R^d\setminus\Omega)$.

 We may integrate (\ref{eqD1}) between $t_1,t_2$, where $t_{i}\in [0,T]$, by virtue of Lemma \ref{aux3}, in order to obtain, 
\[\Big|\int_{\Omega}[\T_{k}(\vr)]_{\ep}^{2}(t_2,\cdot)\eta(x) \dx - \int_{\Omega}[\T_{k}(\vr)]_{\ep}^{2}(t_1,\cdot)\eta(x) \dx\Big|
\]
\[
\leq C\Big( \|\vu\|_{L^{p}(t_1,t_2; W^{1,q}(\Omega))} +\|r_\ep\|_{L^p(t_1,t_2;L^{\widetilde q}(\Omega))}
\| \eta \|_{C^1(\overline\Omega)}\Big)({t_{2}-t_{1}})^{1/p'}, \]
where $C$ may depend on $k$ but is independent of $0<\ep<{\rm dist}({\rm supp}\,\eta, R^d\setminus\Omega)$.
The latter inequality shows in view of Lemmas \ref{prop1}, \ref{l2} that the family of maps (\ref{maps}) 
is for any $k>1$ and $\eta \in C^1_c(\Omega)$  equi-continuous in $C([0,T])$.

Now, we denote $\mathcal{J}(\Omega)\subset C_{c}^{1}({\Omega})$ a countable dense subset of $L^{2}(\Omega)$. 
 Using Arzel\`a--Ascoli theorem and countability of $\mathcal{J}(\Omega)$ (in order to employ a diagonalization procedure) we may show that there is a subsequence of $\ep \to 0$ and $Z^{(k)}_{\eta}\in C([0,T])$ such that $\forall \eta \in \mathcal{J}({\Omega})$
\[\int_{\Omega}[\T_{k}^{2}(\vr)(t,x)]_{\ep}\eta(x)\dx \mapsto Z^{(k)}_{\eta} \;\mbox{ in} \ C[0,T]\;\mbox{as $\ep\to 0+$}.\]
By virtue of (\ref{D0})
\[Z^{(k)}_{\eta}(t)=\int_{\Omega}{\T^2_{k}(\vr)(t,x)}\eta(x)\dx.\]

Now we  use density of $\mathcal{J}({\Omega})$ in $L^{2}(\Omega)$ and the uniform bound with respect to $\ep$ of
$\sup\tau\in [0,T]\|[\T_{k}(\vr)(t,x)]_{\ep}\|_{L^2(\Omega)}$ (cf. the last inequality in Lemma \ref{Lbg} and Item 2. in Lemma \ref{prop1})
 to show that
\[
\int_{\Omega}[\T_{k}(\vr)(t,x)]_{\ep}^{2}\eta(x)\dx 
\mapsto \int_{\Omega}{\T^2_{k}(\vr)(t,x)}\eta(x) \dx 
\]
in $C([0,T])$  for  all $\eta \in L^{2}(\Omega)$. In particular, 
\begin{equation}\label{T2mol}
\forall k>1, \;\Big(t \mapsto \int_{\Omega}\T_{k}(\vr)(t,x)^{2}\dx\Big)\in C([0,T]).
\end{equation}

Resuming: According to (\ref{Tmol})
$$
\T_k(\vr(t'))\to\T_k(\vr(t))\;\mbox{weakly in $L^2(\Omega)$ as $t'\to t$}
$$ 
and according to
(\ref{T2mol}),
$$
\|\T_k(\vr(t'))\|_{L^2(\Omega)}\to\|\T_k(\vr(t))\|_{L^2(\Omega)}\;\mbox{as $t'\to t$}.
$$ 
Since weak convergence and convergence in norms in $L^2(\Omega)$ imply strong convergence, we have
$$
\T_{k}(\vr)\in C([0,T];L^{2}(\Omega))\;\mbox{ for any $k>1$}.
$$

It remains to  show that the latter formula implies $\vr \in C([0,T],L^{r}(\Omega))$, $1\le r<\gamma$. To this end, we write
$$
\sup_{t\in[0,T]}\|(\T_{k}(\vr)-\vr)(t)\|_{L^{r}(\Omega)}\leq\|T_{k}(\vr)-\vr\|_{L^\infty(0,T;L^{r}(\Omega))},
$$
where we have used the last inequality in Lemma \ref{Lbg}. Consequently, for all $t\in [0,T]$,
\[
\|(\T_{k}(\vr)-\vr)(t)\|^r_{L^{r}(\Omega)}\leq { {\rm ess \, sup}}_{t\in (0,T)}
\int_{\{|\vr|\geq k\}}2^r|\vr|^r\dx
\]
\[ 
\leq 2^r{ {\rm ess \, sup}}_{t\in (0,T)}\Big[\Big(\int_{\Omega}|\vr|^{\gamma}\dx\Big)^{\frac{r}{\gamma}}|{\{|\vr|\geq k\}}|^{\frac{\gamma-r}{\gamma}}\Big],
\]
where
$$
|{\{|\vr|\geq k\}}|\le\frac 1 k\int_{\{|\vr|\geq k\}}|\vr|\, {\rm d} x\le \frac 1k |{\{|\vr|\geq k\}}|^{1/\gamma'}\|\vr\|_{L^\gamma(\Omega)}.                                                                                                     
$$
Whence,
$$
\forall t\in [0,T], \; \|\T_{k}(\vr)-\vr\|_{L^{r}(\Omega)}\to 0\;\mbox{ as $k\to\infty$}.
$$
With this information, writing,
  \[\|\vr(t)-\vr(t')\|_{L^{r}(\Omega)}\leq \|\vr(t)-\T_{k}(\vr)(t)\|_{L^{r}(\Omega)}+\|\T_{k}(\vr)(t')-\T_{k}(\vr)(t)\|_{L^{r}(\Omega)}\]
  \[+\|\T_{k}(\vr)(t')-\vr(t')\|_{L^{r}(\Omega)},
 \]
we conclude that $\vr\in C([0,T]; L^r(\Omega))$.
 \end{proof}

 \section{Proof of Theorems \ref{th3}--\ref{th4}}\label{S6}

It is enough to outline the proof only in the case of Theorem \ref{th3}. The proof of Theorem \ref{th4} follows the same lines.

The proof of Statements 1.1 and 2.1   of Theorem \ref{th3}  is based on regularization of the equation via mollifiers, 
cf. Lemma \ref{prop1}. The regularized equation 
$$
\partial[\vr]_\ep+{\rm div}([\vr]_\ep\vu)=r_\ep(\vr,\vu),\qquad r_\ep(\vr,\vu)={\rm div}([\vr]_\ep\vu)-{\rm div}[\vr\vu]_\ep
$$
{ is satisfied}  almost everywhere
in $I\times\Omega_\ep$, $\Omega_\ep=\{x\in \Omega\,|\, {\rm dist}(x,R^d\setminus\Omega)>\ep\}$ and can be therefore multiplied by $b'([\vr]_\varepsilon)$. The Friedrichs commutator lemma (cf. Lemma \ref{l2}) ensures that the term $r_\ep\to 0$ in $L^1(I; L^1_{\rm loc}(\Omega))$. It is the main property which allows to conclude
at the first stage for $b$ in class (\ref{ren}), and consequently, for any $b$ in class (\ref{t1.3}), by using a convenient
approximation of the function $b$ in class  (\ref{t1.3})  and the dominated Lebesgue convergence theorem.  This is the standard
procedure introduced in the same context in the seminal work  \cite{DL}. 

Concerning the proof of Statements 1.2 and 2.2 of Theorem \ref{th3}, we shall show  solely the { latter.} 
Furthermore, it is enough to  deal only with the "integrability up to $\partial \Omega$" in the case of { Statement 2.2.}

We define a function $\xi_n$ as follows:
\begin{equation*}
\xi_{n}(x) :=\chi_{n}({\rm dist}\,(x,\partial\Omega))
\end{equation*}
with
\begin{equation*}
\chi_{n}(s) = \chi(ns),
\end{equation*}
where
\begin{equation*}
\chi\in C^\infty([0,\infty)),\quad 0\le \chi',\quad
\chi(s) = 
 \begin{cases}
   0 &\text{if $0\le s\leq \frac{1}{4} $}\\
    1 &\text{if $s \geq \frac{1}{2}$.}\\
 \end{cases}
\end{equation*}
Recall that ${\rm dist}(\cdot, \partial\Omega)$ is a 1-Lipschitz function. 


Notice that it can be deduced from the above 
\begin{equation*}
\xi_n\in C^\infty([0,\infty)),\;\xi_n'(x)\le Cn,\;
\xi_{n}(x) = 
 \begin{cases}
   0 &\text{if ${\rm dist}(x,\partial\Omega)\leq \frac{1}{4n} $}\\
   1 &\text{if ${\rm dist}(x,\partial\Omega)\geq \frac{1}{2n},$}\\
 \end{cases}
\end{equation*}
with some $C>0$ ($C$ depends on the  choice of $\chi$).

 We calculate for $\eta \in C^\infty(\overline{\Omega})$
\[\int_{\Omega}\vr(\tau, x)\psi(\tau)\eta (x)\dx-\int_{\Omega}\vr(0,x)\psi(0)\eta(x)\dx-\int_{Q}\vr(t,x) \partial_{t}\psi(t)\eta(x)\dxdt\]
\[-\int_{Q}\vr(t,x) \vu(t,x)\cdot \nabla \eta(x) \psi(t)\dxdt\]
\[=\int_{\Omega}\vr(\tau, x)\psi(\tau)\eta (x)\xi_{n}(x)\dx-\int_{\Omega}\vr(0,x)\psi(0)\eta(x)\xi_{n}(x)\dx
\]
\[
-\int_{Q}\vr(t,x) \partial_{t}\psi(t)\eta(x)\xi_{n}(x)\dxdt\]
\[-\int_{Q}\vr(t,x) \vu(t,x)\cdot \nabla \big(\eta(x) \xi_{n}(x)\big)\psi(t)\dxdt
\]
\[+\int_{\Omega}\vr(\tau, x)\psi(\tau)\eta (x)(1-\xi_{n}(x))\dx-\int_{\Omega}\vr(0,x)\psi(0)\eta(x)(1-\xi_{n}(x))\dx\]
\[-\int_{Q}\vr(t,x) \partial_{t}\psi(t)\eta(x)\big(1-\xi_{n}(x)\big){ \dxdt }
\]
\begin{equation}\label{eqxi}
-\int_{Q}\psi(t)\vr(t,x)\vu(t,x)\cdot \nabla \big(\eta(x) (1-\xi_{n}(x))\big)\dxdt.
\end{equation}

We easily verify due to the above formulas for $\xi_n$ that $\eta \xi_{n} \in W_{0}^{1,p}(\Omega)$ with any $1\leq p<+\infty$. Since $C_{{\rm c}}^{1}(\Omega)$ 
is dense in $W_{0}^{1,p}(\Omega)$, it is an admissible test function for equation (\ref{co1.1}). Consequently, the sum of first four terms at the right hand side
(terms containing $\eta \xi_{n} $)  is equal to $0$.

To complete the proof we would like to show that the limit $n \to +\infty $ of the sum of the last four terms at the right hand side of identity (\ref{eqxi}) is zero. To this aim we have to assume that all functions are integrable up to the boundary of $\Omega$.
 
We set $A_{n}:=\{x: {\rm dist}(x,\partial\Omega)\leq \frac{1}{2n}) \}$. Since $\Omega$ is a bounded Lipschitz domain, $|A_{n}|\to 0$.
In the sequel, we will systematically use this fact.

We have
\begin{enumerate}
\item
\[\int_{Q}|\vr(t,x)\partial_{t}\psi(t)\eta(x)(1-\xi_{n}(x))|\dxdt\]
\[=\int_{0}^{T}\int_{A_{n}}|\vr(t,x)\partial_{t}\psi(t)\eta(x)|\dxdt\]
\[\leq C\|\vr\|_{L^{\alpha}(0,T;L^\beta(A_n))}\|\partial_{t}\psi\|_{L^{\infty}((0,T))}\|\eta\|_{L^{\infty}(\Omega)}|A_{n}|^{1-\frac{1}{\beta}}\to 0,\quad n\to \infty.\]
\item
\[\int_{Q}|\vr(t,x)\vu(t,x)\cdot \nabla \big(\eta(x)(1-\xi_{n}(x))\big)\psi(t)|\dxdt\to 0,\quad n \to \infty, \]
where we have used the fact that $\vu \in L^{p}(I,W^{1,q}_{0}(\Omega;R^d))$. Indeed,
\[\lim_{n\to \infty}\int_{Q}\big|\vr(t,x)\vu(t,x)\cdot\nabla\big(\eta(x)(1-\xi_{n}(x))\big)\psi(t)\big|\dxdt\]
\[\leq \lim_{n\to \infty}\int_{Q}\big|\vr(t,x)\vu(t,x)\cdot \nabla \eta(x)\big(1-\xi_{n}(x)\big)\psi(t)\big|\dxdt
\]
{ \[
+\lim_{n\to \infty}\int_{Q}\big|\vr(t,x)\vu(t,x)\cdot\nabla \xi_{n}(x) \psi(t)\eta(x) \big|\dxdt\]
\[=\lim_{n\to \infty}\int_{Q}\big|\vr(t,x)\vu(t,x)\cdot\nabla \xi_{n}(x) \psi(t)\eta(x)\big|\dxdt\] }
\[\leq\lim_{n\to \infty}C\int_{0}^{T}\int_{A_{n}}\Big|\vr(t,x)\frac{\vu(t,x)}{{\rm dist}(x,\partial\Omega)}\cdot \nabla {\rm dist}(x,\partial\Omega) \psi(t)\eta(x)\Big|\dxdt\]
 \[\leq \lim_{n\to \infty}C \int_{0}^{T}\Big\|\frac{\vu(t,x)}{{\rm dist}(x,\partial\Omega)}\Big\|_{L^{q}(\Omega)} \|\vr(t)\|_{L^{\beta}(A_{n})}\|\psi\|_{L^{\infty}((0,T))} \|\eta\|_{L^{\infty}(A_{n})} \dt  \]
\[\leq \lim_{n\to \infty}C \int_{0}^{T}\|\vr(t)\|_{L^{\beta}(A_{n})}\|\psi\|_{L^{\infty}([0,T])} \|\eta\|_{L^{\infty}(A_{n})} \|\nabla \vu\|_{L^{q}(\Omega)} \dt\]
\[\leq \lim_{n\to \infty}C\|\vr\|_{L^{\alpha}(0,T;L^\beta(A_n))}\|\psi\|_{L^{\infty}((0,T))} \|\eta\|_{L^{\infty}(A_{n})} \|\nabla  \vu\|_{L^{p}(0,T;L^q(\Omega;R^{d\times d}))}\]
\[=0,\]
after employing the Hardy inequality (hence $\Omega$ must have Lipschitz boundary).
\end{enumerate}
Similarly we treat also the first two integrals over $\Omega${, where we use the fact that $\vr \in C_{{\rm weak}}(\overline{I};L^\gamma(\Omega))$ and the product $\eta(1-\xi_n)$ is bounded uniformly in $L^\infty(\Omega)$. This finishes the proof of Statement 2.2} and thus Theorem \ref{th3} as well as Theorem \ref{th4} are proved.

\section{Proof of Theorem \ref{th5}}\label{S7}

We present the proof for distributional solutions only. { The case of weak solutions follows more or less the same lines. Due to the fact that $\Omega$ is Lipschitz, we may extend the function $\vu$ to the whole $R^d$ in such a way that it belongs to $L^p(I;W^{1,q}(R^d;R^d))$ and either $\vr$ or $s$ by zero outside $\Omega$. Then, clearly, the extended $\vr$ resp. $s$ solve the continuity resp. transport equation in the whole $I\times R^d$ with the transporting velocity the extended $\vu$. We can therefore apply the mollification in $R^d$ and then equations (\ref{4.6}) hold a.e. in $I\times R^d$. Hence we may repeat the whole proof given below in $I\times R^d$.}   


Let us start with Statement 1. Since both $\vu \cdot \nabla \vr$ and $\vu \cdot \nabla s$ fulfill assumptions of the Friedrichs commutator lemma (Lemma \ref{l2}), we see that $[\vr]_\varepsilon$ and $[s]_\varepsilon$, the corresponding mollifications in the spatial variable satisfy a.e. in $I\times \Omega_\varepsilon$, where $\Omega_\varepsilon$ is defined in the proof of Lemma \ref{aux4}, 
\begin{equation} \label{4.6}
\begin{aligned}
\partial_t [s]_\ep + \vu \cdot \nabla [s]_\ep &= r_\ep^1, \\
\partial_t [\vr]_\ep + \Div ([\vr]_\ep \vu) &= r_\ep^2,
\end{aligned}
\end{equation}  
where $r_\ep^1\to 0$ in $L^{\tau_1}(I;L^{\sigma_1}_{{\rm loc}}(\Omega))$, $\sigma_1\in [1,q)$ if $\beta_\vr=\infty$, $\frac{1}{\sigma_1} \geq \frac{1}{\beta_\vr}+ \frac{1}{q}$ otherwise,  and $\frac{1}{\tau_1} \geq \frac{1}{\alpha_\vr}+ \frac{1}{p}$, $\tau_1<\infty$. Similarly $r_\ep^2\to 0$ in $L^{\tau_2}(0,T;L^{\sigma_2}(\Omega))$, $\sigma_2\in [1,q)$ if $\beta_s=\infty$, $\frac{1}{\sigma_2} \geq \frac{1}{\beta_s}+ \frac{1}{q}$ otherwise,  and $\frac{1}{\tau_2} \geq \frac{1}{\alpha_s}+ \frac{1}{p}$, $\tau_2<\infty$.  We may multiply \eqref{4.6}$_1$ by $[V]_\ep$ and \eqref{4.6}$_2$ by $[s_\vr]_\ep$. Thus, a.e. in { $I\times \Omega_\varepsilon$,}
$$
\partial_t ([s]_\ep [\vr]_\ep) + \Div([s]_\ep [\vr]_\ep \vu) = r_\ep^1 [\vr]_\ep + r_\ep^2 [s]_\ep,
$$
i.e.
$$
\int_0^T \int_\Omega \Big([s]_\ep [\vr]_\ep\partial_t \varphi + [s]_\ep [\vr]_\ep\vu \cdot \nabla \varphi + (r_\ep^1 [\vr]_\ep + r_\ep^2 [s]_\ep)\varphi\Big)  \dxdt =0
$$
for all $\varphi \in C^\infty_{\rm c} ((0,T)\times\Omega_\ep)$. 

We now intend to let $\ep \to 0+$. We need to verify that the first two terms converge to the corresponding counterparts while the last two terms converge to zero.

First, since the sequence $[s]_\ep$ is bounded in { $L^{\alpha_s}(I;L^{\beta_s}(\Omega_\varepsilon))$,} the term $r_\ep^2 [s]_\ep \to 0$ in $L^1((0,T)\times\Omega)$. Similarly, since $[\vr]_\ep$ is bounded in the space $L^{\alpha_\vr}(I;L^{\beta_\vr}(\Omega_\varepsilon))$,  { the other term also goes to zero.}

Next we consider the first and the second term. Indeed, the second term is more restrictive than the first one. Since $[s]_\ep \to s$ in $L^{t_s}(I;L^{r_s}_{{\rm loc}}(\Omega))$, $[\vr]_\ep \to \vr$ in $L^{t_\vr}(I;L^{r_\vr}_{{\rm loc}}(\Omega))$ and $\vu \in L^p(I;W^{1,q}(\Omega;R^d))$, we easily see $[\vr]_\ep [s]_\ep \vu \to \vr s \vu$ in $L^1(I;L^1_{{\rm loc}}(\Omega;R^d))$. This finishes the proof of Statement 1.

In the case of Statement 2 we first proceed as above and verify that $\vr s$ is a distributional (weak) solution to the continuity equation. Only in the limit passage of $[\vr]_\ep [s]_\ep \vu$ we have to employ additionally the Sobolev embedding theorem for $\vu$ in the spatial variable together with the $L^\infty$ bound in time for  { $[\vr]_\ep$ and $[s]_\ep$ if some of the exponents is equal to $\infty$, and interpolate these bounds. Next, we apply Theorem \ref{th3}, Statement 1.1,} to see that $\vr s$ is a renormalized distributional solution to the continuity equation.


Furthermore, since $\vr s\in C_{{\rm weak}}(\overline{I}; L^\gamma(\Omega))$ and $\gamma >1$, we may employ Theorem { \ref{th1}, Statement 3.,} to verify that $\vr s \in C(\overline{I};L^r(\Omega))$ for any $1\leq r<\gamma$. Theorem \ref{th5} is proved.

\section{Proof of the main results}\label{S8}

\subsection{Proof of Theorem \ref{t2}}

To proof Theorem \ref{t2} we first use the fact that $\vr$ is a renormalized time integrated { weak} solution of the transport equation and use $b_{\delta}(\vr):= \frac{\delta}{\delta + \vr}$ with $\delta >0$ in the renormalized formulation. As we know that $\vr \geq 0$ a.e. in $(0,T)\times \Omega$, the function $b_\delta$ is an appropriate renormalizing function\footnote{{ Strictly speaking, function $b_\delta$ does not satisfy the second condition (\ref{ren}). Nevertheless, the map
$\vr\mapsto\vr b_\delta'(\vr)-b_\delta(\vr)$
remains bounded. We can thus take instead of $b_\delta$ a convenient approximation (e.g. $j_\ep*\max\{b_\delta(\cdot+\ep),1/\ep\}$, $\ep\in (0,\delta)$,
see Lemma \ref{prop1} for the notation) which satisfies (\ref{ren}), and then let $\ep\to 0$ in order to get (\ref{4.1}).}}.  We get 
\begin{equation} \label{4.1}
\begin{aligned}
&\int_\Omega \frac{\delta}{\delta +\vr(t,\cdot)} \varphi(t,\cdot) \dx - \int_\Omega \frac{\delta}{\delta +\vr(0,\cdot)} \varphi(0,\cdot) \dx -\int_0^t\int_\Omega \frac{\delta}{\delta +\vr} \partial_t \varphi \dx \, {\rm d}\tau \\
= & \int_0^t\int_\Omega \Big(\frac{\delta}{\delta +\vr} \vu \cdot \nabla \varphi + \Big(\frac{\delta}{\delta +\vr}-\frac{\delta \vr}{(\delta +\vr)^2}\Big)\Div \vu \Big)\dx \, {\rm d}\tau 
\end{aligned}
\end{equation}
for all $\varphi \in C_{\rm c}^\infty([0,T]\times \overline{\Omega})$.  We may let $\delta \to 0+$  in \eqref{4.1}  to get (we use the Lebesgue dominated convergence theorem; recall that $\frac{\delta}{\delta +\vr(t,x)} =1$ provided $\vr(t,x)=0$)
\begin{equation} \label{4.3}
\begin{aligned}
&\int_\Omega s_\vr(t,\cdot) \varphi(t,\cdot) \dx - \int_\Omega s_\vr(0,\cdot) \varphi(0,\cdot) \dx - \int_0^t \int_\Omega s_\vr \partial_t \varphi \dx \, {\rm d}\tau \\
  =  & \int_0^t \int_\Omega  \big(s_\vr \vu\cdot \nabla \varphi + s_\vr \Div \vu \varphi\big) \dx \, {\rm d}\tau  
\end{aligned}
\end{equation}
for all $\varphi$ as above. Here, $s_\vr$ denotes the characteristic function of the set, where $\vr=0$. 
Hence $s_\vr$ is a  time integrated weak solution to the transport equation with the function $\vu$. { Moreover, repeating the argument above with $\widetilde {b}(\vr):= b\Big(\frac{\delta}{\delta +\vr}\Big)$, where $b$ belongs to the class (\ref{ren}), we also get that $s_\vr$ is a renormalized time integrated weak solution.}


Since $\int_\Omega s_\vr(\tau,\cdot) \dx = |\{x\in \Omega; \vr(\tau,x)=0\}|_d$, we may subtract equations \eqref{4.3} with $\varphi = 1$ for $t:=\tau_1$ and $t:=\tau_2$ and it is easy to see that
$$
\Big|\int_{\tau_1}^{\tau_2} \int_\Omega s_\vr \Div \vu \dxdt\Big| \to 0 \qquad \text{ for } \tau_1\to \tau_2.
$$
Hence 
$$
|\{x\in \Omega; \vr(\tau,x)=0\}|_d \in C([0,T]).
$$
Note further that repeating the argument to get  \eqref{4.3} with a test function only space dependent, we get $s_\vr \in C_{\rm weak}([0,T];L^r(\Omega))$ for any $1\leq r<\infty$ and thus, by { Lemma \ref{aux4},} 
$$
s_\vr \in C([0,T];L^r(\Omega)), \qquad 1\leq r<\infty.
$$
The theorem is proved.

\subsection{Proof of Theorem \ref{t3} and Corollaries \ref{cor1}--\ref{cor2}}

The first claim of { Theorem \ref{t3}} is a direct consequence of Theorems \ref{th1} and \ref{th3}. The second claim  follows directly from Theorem { \ref{th5}, Statement 2.} The third claim is a direct consequence of formula (\ref{clRR}).

Corollary \ref{cor1} follows immediately from Theorem \ref{t3}. 

We now consider Corollary \ref{cor2}. We aim at proving that $|\{x\in \Omega|\vr(t_0,x)=0\}|_d =0$ for any $t_0 \in (0,T]$. First, for $t_0 \in (0,\tau]$, we define $\widetilde R(t):= R(t-t_0+\tau)$. Since $\vu$ is time independent, the function $\widetilde R$ is a distributional solution to the continuity equation on { $(t_0-\tau,T'+t_0-\tau)$} and we may apply Theorem \ref{t3}, in particular formula (\ref{clRR}) with $\widetilde R$ instead of $R$ and $t_0$ instead of $t$. Hence Corollary \ref{cor2} holds in the time interval $(0,\tau]$. Next we consider { $t_0\in(\tau,2\tau]$.} We redefine $\widetilde R$ as { $\widetilde R(t) :=R(t-\tau)$} and { apply} formula (\ref{clRR}) with $t:=t_0+\tau$ on the {left hand side} and $\tau$ instead of $0$ on { the right hand side. Hence $|\{x\in \Omega|\vr(t,x)=0\}|_d=0$ for $t \in [0,2\tau]$. Proceeding similarly,
after} finite number of steps we cover the whole interval $(0,T)$. Corollary \ref{cor2} is proved.  

Note finally that due to our definition of the weak solution to the compressible Navier--Stokes equations Corollary \ref{c1} follows directly, since all assumptions of Theorems \ref{t2} and \ref{t3}, as well as { of} Corollaries \ref{cor1} and { \ref{cor2},} are fulfilled.

{\bf Acknowledgement:} The work of M. Pokorn\'y was supported by the Czech Science Foundation, grant No. 16-03230S. Significant part of the paper was written during the stay of M. Pokorn\'y at the University of Toulon. The authors acknowledge this support.  

{\bf Conflict of interest:} The authors declare that they have no conflict of interest.

\end{document}